\mag=\magstep1 
\input epsf 
\documentclass[10pt,a4paper]{amsart} 
\usepackage{amssymb,latexsym} 
\usepackage[dvipdfm]{graphicx,color}

\def\Cal{\mathcal}

\def\ot{\leftarrow} 
\def\<<{\langle } 
\def\>>{\rangle }

\numberwithin{equation}{section} 
\newtheorem{theorem}{Theorem}[section] 
\newtheorem{proposition}[theorem]{Proposition} 
\newtheorem{corollary}[theorem]{Corollary} 
\newtheorem{definition}[theorem]{Definition} 
 
\newtheorem{remark}[theorem]{Remark} 
\newtheorem{lemma}[theorem]{Lemma} 
 
\newtheorem{example}[theorem]{Example}

\newtheorem{notation}[theorem]{Notation}

\def\<{\langle} 
\def\>{\rangle}

\begin{document} 
 
\title{
Brown-Zagier relation for associators} 
 
\author{Tomohide Terasoma}

\maketitle 
\makeatletter





\section{Introduction}
We have big heritage of equalities on hypergeometric funcitons, which can be
used for showing many equalities for multiple zeta values.
This method can be also applicable for showing relations
between coefficients of associators using
the theory of $\Phi$-cohomology. 
A $\Phi$-cohomology is 
equipped with two realizations
$B,dR$ and a comparison map described by the given
associator $\Phi$.
Brown \cite{B} used certain relation between multiple zeta values
to show the injecctivity of the homomorphism from
Motivic Galois group to Grothendieck-Teichmuller group.
This relation was proved by Zagier \cite{Z}, which we call Brown-Zagier
relation.
After his work, Li \cite{L} gave another proof of Brown-Zagier relation
using several functional equations of hypergeometric series.

In this paper, we show that Brown-Zagier relation holds also
for the coefficietns of any associators.
In the paper \cite{L}, he proved Brown-Zagier relation using
Dixon's theorem which is equivalent
to Selberg integral formula. The Selberg integral formula
arises from symmetric product construction, which does
no exist in the category of moduli space. Even in this case,
we can construct isomorphism between $\Phi$-local systems
using descent theory.
The main theorem is steted as follows.
\begin{theorem}
\label{main theorem}
We use the notation for coefficients 
$\zeta_{\Phi}(n_1, \dots, n_m)$ of an associator $\Phi$.
Then we have
$$
\zeta_{\phi}(2^a,3 ,2^b)=2\sum_{r=1}^{a+b+1}(-1)^{r}c_{a,b}^r
\zeta_{\Phi}(2r+1))\zeta_{\Phi}(2^{a+b-r+1}),
$$
where
$$
c_{a,b}^r=
\left(
\begin{matrix}
2r \\ 2a+2
\end{matrix}
\right)-(1-\frac{1}{2^{2r}})\left(
\begin{matrix}
2r \\2b+1
\end{matrix}
\right)
$$
\end{theorem}
Since the generating series of motivic multiple zeta values satisfies the
associator relation, we have the following corollary.
\begin{corollary}
The same relation holds in the coordinate ring of mixed Tate motives.
\end{corollary}
The above corollary gives an another proof of a result of Brown.
\begin{notation}
The product of Gamma function $\Gamma(a_1)\Gamma(a_2)\cdots \Gamma(a_n)$
is denoted by $\Gamma(a_2,a_2,\dots, a_n)$ for short.
\end{notation}
\section{Differential equations and generating functions}
In this section, we recall outline of classical theory of
hypergoemtric functions and Gauss-Manin connections.

\subsection{Differential equation and iterated integral}
Let $N^*$ be a $\bold C\<\<e_0,e_1\>\>$ left module.
The action of $e_0$ and $e_1$ on $N^*$ is denoted by $P_0$ and $P_1$,
Let $\Cal O_U$ be the ring of analytic functions on an open set $U$ in 
$\bold P^1-\{0,1,\infty\}$.
We define a map 
$P:N^*\otimes O_{an}\to N^*\otimes \Omega^1_{an}$
$$
N^* \to N^*\otimes \<\frac{dx}{x},\frac{dx}{x-1}\>:
v\mapsto P(x)v 
$$
where $P(x)=P_0\frac{dx}{x}+P_1\frac{dx}{x-1}$.
There exists a unique local solution $\Phi_u(x)$
of the differential equation $d\Phi(x)=P(x)\Phi(x)$
for $End(N^*,N^*)$-valued analytic functions
such that $\Phi_u(u)=id_V$.
It is denoted by $\displaystyle\exp(\int_u^xP)$.
For any solution of the differential equation $dV=PV$ for 
$End(N^*,N^*)$-valued
functions, we have
$$
V(x)=(I+\int_{u}^xP+\int_{u}^xPP+\cdots)V(u)=\exp(\int_{u}^xP)V(u)
$$
for $t\in \bold R, 0<t_0<\epsilon$.
Fora path $\gamma$ from $u$ to $u'$,
$\exp(\int_u^{u'}P)$ depends only on the homotopy class of $\gamma$,
which is denoted by
$\rho(\gamma)$. Then $\rho$ defines a left $\pi_1(\Cal M_4)$-module on $N^*$.

\subsection{Differential equation of Gauss hypergeometric functions}
\label{subsec: Differential equation GM connection}
In this section, we recall the differential equations 
satisfied by hypergoemtric functions.
\subsubsection{}
We define hypergeometric function by
\begin{align*}
F(a,b;c;x)=&\sum_{n=0}^{\infty}\frac{(a)_n(b)_n}{n!(c)_n}x^n, 
\end{align*}
where $(a)_n=a(a+1)\cdots (a+n-1)$. The hypergoemetric function has
the following integral expression.
\begin{align*}
B(a,c-a)F(a,b;c;x)
=\int_0^1t^{a-1}(1-t)^{c-a-1}(1-xt)^{-b}dt.
\end{align*}
\subsubsection{}
The differential is denoted by $D=\frac{\partial}{\partial x}$.
We define a matrix $V_0=(v_{ij})_{1\leq i,j \leq 2}$, where
\begin{align*}
 v_{11}^{(0)}=&\frac{\Gamma(a,c-a+1)}{\Gamma(c+1)}F(a,b;c+1;x),\quad
\\
v_{12}^{(0)}=&x^{- c} 
\frac{\Gamma(b-c,1-b)}{\Gamma(1-c)}
F(b - c, a - c; 1 - c;x),\quad 
\\
v_{22}^{(0)}=&\frac{1}{a}xD(v_{12}),\quad
v_{21}^{(0)}=\frac{1}{a}xD(v_{11})
\end{align*}
Let $P$ be a matrix defined by
\begin{equation}
\label{generic differential operator}
P=\frac{dx}{x}
P_0
+
\frac{dx}{x-1}
P_1
\end{equation}
where
\begin{equation*}
P_0=
\left(
\begin{matrix}
0 & a \\
0 & -c
\end{matrix}\right),
P_1=
\left(\begin{matrix}
0 & 0 \\
-b & c-a-b
\end{matrix}
\right).
\end{equation*}
Then the matrix $V_0$ satisfies
the differential equation
\begin{equation}
\label{hyper geometric Pfaff equation }
dV=PV.
\end{equation}
Let $V_1$ be matrix defined by
\begin{align*}
 v_{11}^{(1)}=&\frac{\Gamma(a,b-c)}{\Gamma(a+b-c)}F(a,b;a+b-c;1-x),\quad
\\
v_{12}^{(1)}=&
(1-x)^{c-a-b+1}
\frac{\Gamma(c+1-a,1-b)}{\Gamma(c+2-a-b)}
F(c+1-a,c+1-b;2+c-a-b;1-x),\quad \\
v_{22}^{(1)}=&\frac{1}{a}xD(v_{12})\quad
v_{21}^{(1)}=\frac{1}{a}xD(v_{11}), 
\\
\end{align*}
Then $V_1$ also satisfies the differential equation 
(\ref{hyper geometric Pfaff equation }).

\subsubsection{Connections and differential equations for coefficients}
Let $N^*$ be the vector space generated by $\omega_1^*, \omega_2^*$,
$N$ be its dual and $\omega_1, \omega_2$ be the dual basis of
$\omega_1, \omega_2$.
We define a linear maps $\nabla:N\to N\otimes 
\<\frac{dx}{x},\frac{dx}{x-1}\>$ and 
$\nabla^*:N^*\to N^*\otimes \<\frac{dx}{x},\frac{dx}{x-1}\>$ by
\begin{align}
\label{gauss manin connection as matrix}
&\nabla
\left(
\begin{matrix}
\omega_1 \\ \omega_2
\end{matrix}
\right)=
P\left(
\begin{matrix}
\omega_1 \\ \omega_2
\end{matrix}
\right) 
\\
\nonumber
&
\nabla^*
\left(
\begin{matrix}
\omega_1^* & \omega_2^*
\end{matrix}
\right)=
-\left(
\begin{matrix}
\omega_1^* & \omega_2^*
\end{matrix}
\right)P
\end{align}
The map $\nabla^*$ can be extended to a connection on the
$N^*\otimes \bold C[x,\frac{1}{x},\frac{1}{x-1}]$, 
which is also denoted by $\nabla^*$.
We use the following identification
$$
f_1(x)\omega_1^*+f_2(x)\omega_2^*=\left(
\begin{matrix}
f_1(x) \\ f_2(x) 
\end{matrix}
\right)
$$
Then we have
$$
\nabla\left(
\begin{matrix}
f_1(x)dx \\ f_2(x)dx 
\end{matrix}
\right)
=d
\left(
\begin{matrix}
f_1(x) \\ f_2(x) 
\end{matrix}
\right)
-P
\left(
\begin{matrix}
f_1(x) \\ f_2(x) 
\end{matrix}
\right)
$$
on $N^*\otimes \bold C[x,\frac{1}{x},\frac{1}{x-1}]$.
Let $\gamma$ be an $N^*$-valued analytic function.
Using the pairing $\<\ ,\ \>$, $\gamma$
is written as
\begin{align*}
&\gamma=
\<\gamma,\omega_1\>\omega_1^*+
\<\gamma,\omega_2\>\omega_2^*=
\left(\begin{matrix}
\<\gamma,\omega_1\> \\
\<\gamma,\omega_2\> \\
\end{matrix}
\right),
\end{align*}
Therefore $\gamma$ is a horizontal section for $\nabla$ if and only if
$$
d\left(\begin{matrix}
\<\gamma,\omega_1\> \\
\<\gamma,\omega_2\> \\
\end{matrix}
\right)=P
\left(\begin{matrix}
\<\gamma,\omega_1\> \\
\<\gamma,\omega_2\> \\
\end{matrix}
\right)
$$


\subsubsection{Hypergeometric function and its integral expression}
By the integral expression of hypergoemetric functions, 
the matrix elements of $V_0$ is written as 
\begin{align}
\label{integral expression matrix}
&v_{11}^{(0)}=\int_{\gamma_1}\omega_1,
v_{12}^{(0)}=\int_{\gamma_2}\omega_1,
v_{22}^{(0)}=\int_{\gamma_1}\omega_2,
v_{22}^{(0)}=\int_{\gamma_2}\omega_2,
\\
\nonumber
&v_{11}^{(1)}=\int_{\gamma_1^\#}\omega_1,
v_{12}^{(1)}=\int_{\gamma_2^\#}\omega_1,
v_{22}^{(1)}=\int_{\gamma_1^\#}\omega_2,
v_{22}^{(1)}=\int_{\gamma_2^\#}\omega_2,
\end{align}
where $\omega_1,\omega_2$ are 
the relative twisted de Rham cohomology classes 
defined by
\begin{align}
\label{hypergeometric base for de Rham}
&\omega_1=\big[\frac{dt}{t}\big], \quad \omega_2=\big[\frac{bxdt}{a(1-xt)}
\big],
\end{align}
and $\gamma_1,\gamma_2$ be twisted cycle defined by
\begin{align}
\label{hypergeometric base for de Rham}
&\gamma_1
=\big[t^a(1-t)^{c-a}(1-xt)^{-b}\big]_{[0,1]}, 
\\
\nonumber
&\gamma_2=\big[t^a(t-1)^{c-a}(xt-1)^{-b}\big]_{[\frac{1}{x},\infty]}
\\
\nonumber
&\gamma_1^\#=\big[(-t)^a(1-t)^{c-a}(1-xt)^{-b}\big]_{[-\infty,0]}
\\
\nonumber
&\gamma_2^\#=\big[t^a(t-1)^{c-a}(1-xt)^{-b}\big]_{[1,\frac{1}{x}]}
\end{align}
We have the following equality of cycles.
\begin{align}
\label{cycles on M_4}
\bold s(c)\gamma_1^\#=&
\bold s(c-a)\gamma_1
+\bold s(b)
\gamma_2, \\
\nonumber
\bold s(c-a-b)\gamma_1=&
\bold s(c-b)
\gamma_1^\#+\bold s(b)\gamma_2^\#.
\end{align}
where 
$\displaystyle\bold s(z)=\frac{1}
{\Gamma(z)\Gamma(1-z)}=\frac{\sin(\pi z)}{\pi}$.
Let $N$ be the vector space generated by $\omega_1,\omega_2$.
Then the Gauss-Manin connection is given by the map
(\ref{gauss manin connection as matrix}).
Under the comparison map, $\gamma_1, \gamma_2$
defines a horizontal $N^*$-valued analytic map on $(0,1)$. By the expression 
\ref{integral expression matrix},
we have
\begin{align*}
\lim_{\epsilon \to +0}\gamma_1(\epsilon)=&
B(a,c-a+1)\left(
\begin{matrix}
1 \\ 0\end{matrix}
\right), \\
\lim_{\epsilon \to +0}(\epsilon^c\gamma_2)(\epsilon)=&
%
B(b-c,1-b)\left(
\begin{matrix}
1 \\ -\frac{c}{a}\end{matrix}
\right).
\end{align*}

\subsection{Gauss-Manin connection and horizontal section on the daul}
\label{sec:Gauss-Manin connection and horizontal section on the daul}

\subsubsection{Dual differential equations}
\label{subsec: Dual differential equation GM connection}

We construct solutions of 
the dual differential equation around $1$.
We define a matrix $W_1=(w_{ij})_{1\leq i,j \leq 2}$
by
\begin{align*}
&w_{11}=\frac{\Gamma(-a,c-b+1)}{\Gamma(c-a-b+1)}
F(-a,-b;c-a-b+1;1-x),\\
&w_{21}=\frac{\Gamma(a-c,b+1))}{\Gamma(a+b-c+1)}(1-x)^{-c+a+b}F(a-c,b-c;1-c+a+b;1-x)
\\
&w_{12}=-\frac{1}{b}(1-x)D(w_{11}),\quad
w_{22}=-\frac{1}{b}(1-x)D(w_{21}).
\end{align*}
The matrix elements of $W_1$ is written as 
$$
v_{11}=\int_{\gamma_1^*}\omega_1^*,
v_{12}=\int_{\gamma_2^*}\omega_1^*,
v_{22}=\int_{\gamma_1^*}\omega_2^*,
v_{22}=\int_{\gamma_2^*}\omega_2^*,
$$
where $\omega_1^*,\omega_2^*$ are 
\begin{align}
\label{hypergeometric base for de Rham}
&\omega_1^*=\big[\frac{dt}{t}\big], \quad \omega_2^*=\big[\frac{(x-1)dt}{(1-xt)}
\big],
\end{align}
and $\gamma_1^*,\gamma_2^*$ are
\begin{align*}
&\gamma_1^*=\big[(-t)^{-a}(1-t)^{-c+a}(1-xt)^{b}\big]_{[-\infty,0]}
\\
&\gamma_2^*=\big[t^{-a}(t-1)^{-c+a}(1-xt)^{b}\big]_{[1,\frac{1}{x}]}
\end{align*}
Then the matrix $W_1$ satisfies the differential equation
$dW_1=-W_1P$.
Therefore we have
$$
W_1=W_1(t_1)\exp(\int_{x}^{t_1}P).
$$

\subsubsection{Duality and exponential map around $1$}
Let $V_1$ and $W_1$ be matrix defined 
in \S \ref{subsec: Differential equation GM connection},
\S \ref{subsec: Dual differential equation GM connection}.
Since
$$
\frac{\partial}{\partial x}(W_1V_1)=
\frac{\partial W_1}{\partial x}V_1+
W_1\frac{\partial V_1}{\partial x}=-W_1PV_1+W_1PV_1=0
$$
the matrix
$W_1V_1$ does not depends on $x$.
By considering the limit for $x\to 0$, we have
\begin{align}
\label{topological cycle givein exponential}
W_1V_1&=
\left(
\begin{matrix}
\frac{\bold s(a+b-c)}{a\bold s(-a)\bold s(b-c)} & 0 \\
0 & \frac{\bold s(c-a-b)}{a\bold s(a-c)\bold s(b)}
\end{matrix}
\right)=D_1
\end{align}
and as a consequence, we have
$$
V_1(y)D_1^{-1}W_1(x)=\exp\int_x^yP.
$$

\subsection{Generating function of multiple zeta values $\zeta(2,\dots, 3,\dots, 2)$}
We specialize to the case $c=0,a=-b$.
Then the matrix $P_0$, $P_1$
of (\ref{generic differential operator}) becomes
\begin{equation}
\label{connection matrix 2xn}
P_0=\left(
\begin{matrix}
0 & a \\
0 & 0
\end{matrix}\right),
P_1=
\left(\begin{matrix}
0 & 0 \\
a & 0
\end{matrix}
\right).
\end{equation}
Using the limit computation of the last subsection, we have
$$
\left(
\begin{matrix} 
\bold s(a)
 & 0
\end{matrix}
\right)
V_0(x)
\left(
\begin{matrix} 1 \\ 0
\end{matrix}
\right)
=
\left(
\begin{matrix} 1 & 0
\end{matrix}
\right)
\exp(\int_{0}^{x}P)
\left(
\begin{matrix} 1 \\ 0
\end{matrix}
\right)
=F(a,-a;1;x)
$$
\begin{proposition}
\label{formal repetition of 0,1}
Let $P_0$ and $P_1$ be matrices defined in (\ref{connection matrix 2xn}).
\begin{enumerate}
\item
For $I=(i_1, \dots, i_n)\in \{0,1\}^n$, we define
$E_I=P_{i_1}\cdots P_{i_n}$. Then we have
$$
(1,0)P_I
\left(\begin{matrix}1 \\ 0 \end{matrix}\right)
=
\begin{cases}
a^{2n} &\text{ if }I=(10)^n \\
0 &\text{ otherwise. }
\end{cases}
$$
\item
Let $\varphi(e_0,e_1)$ be an element $\bold C\<\<e_0,e_1\>\>$
given by
\begin{equation}
\label{typical element}
\varphi(e_0,e_1)
=\sum_{n\geq 0}\sum_{i_1\in\{0,1\},\dots, i_n\in\{0,1\}}
c_{i_1, \dots, i_n}
e_{i_1}\dots e_{i_n}
\end{equation}
where $c_{i_1,\dots, i_n}\in \bold C$. Then we have
\begin{equation}
\label{2,..2 generating matrix}
(1,0)\varphi(P_0,P_1)
\left(\begin{matrix}1 \\ 0 \end{matrix}\right)
=
1+
\sum_{n> 0}c_{(01)^{n}}a^{2n}.
\end{equation}
\end{enumerate}
\end{proposition}
\begin{proof}
This is an easy consequence of the equalities
$P_0^2=P_1^2=0$, and
$$
P_0P_1=\left(
\begin{matrix}
a^2 & 0 \\
0 & 0
\end{matrix}\right),\quad
P_1P_0=\left(
\begin{matrix}
0 & 0 \\
0 & a^2
\end{matrix}\right).
$$
\end{proof}
Since $V(x)$ is expressed by using iterated integral, we have
\begin{align*}
F(a,-a;1;x)
=&\left(
\begin{matrix} 1 & 0
\end{matrix}
\right)
\sum_{n=0}^{\infty}\int_0^{x}(P_0\frac{du}{u}+P_1\frac{du}{u-1})^n 
\left(
\begin{matrix} 1 \\ 0
\end{matrix}
\right)
\\
=&\sum_{n=0}^{\infty}\int_0^{x}(\frac{du}{u}\frac{du}{u-1})^n a^{2n}
\end{align*}
by Proposition \ref{formal repetition of 0,1}.
By setting $x=1$,
we have
$$
\sum_{n=0}^{\infty}\int_0^{1}(\frac{du}{u}\frac{du}{u-1})^n a^{2n}=
F(a,-a;1;1)=
\frac{\sin(\pi a)}{\pi a}
$$
by the equality (\ref{duality coupling}).
Similarly, we have
$$
B(-a,a+1)^{-1}\left(
\begin{matrix} 1 & 0
\end{matrix}
\right)
W_1(x)
\left(
\begin{matrix} 1 \\ 0
\end{matrix}
\right)
=
\left(
\begin{matrix} 1 & 0
\end{matrix}
\right)
\exp(\int_{x}^1P)
\left(
\begin{matrix} 1 \\ 0
\end{matrix}
\right)
=
F(-a,a,1,1-x),
$$
and it is equal to
$$
\sum_{m=0}^{\infty}\int_x^1(\frac{du}{u}\frac{du}{u-1})^m a^{2m}
$$

We have the following proposition.
\begin{proposition}
We set
\begin{equation}
\label{Gauss-Manin generating function}
\phi(a,b)=
\sum_{m=0}^{\infty}
\sum_{n=1}^{\infty}
\int_0^1(\frac{du}{u}\frac{du}{u-1})^m
\frac{du}{u}
(\frac{du}{u}\frac{du}{u-1})^n a^{2n}b^{2m}
\end{equation}
Then we have
$$
\phi(a,b)=\int_0^1
F(-b,b,1,1-w)
(F(a,-a,1,w)-1)
\frac{dw}{w}
$$
\end{proposition}
\begin{proof}
By the definition of iterated integral,
we have
\begin{align*}
&\int_0^1(\frac{du}{u}\frac{du}{u-1})^m
\frac{du}{u}
(\frac{du}{u}\frac{du}{u-1})^n \\
=&\int_0^1
\bigg[\int_x^1
(\frac{du}{u}\frac{du}{u-1})^m \bigg]
\frac{dx}{x}
\bigg[\int_1^x (\frac{dv}{v}\frac{dv}{v-1})^n 
\bigg]dx.
\end{align*}
By taking the generating function on $m$ and $n$, we get
the proposition.
\end{proof}
\begin{remark}
Zagier showed that $\phi(a,b)$ is also equal to
\begin{equation}
\label{Zagier's expression}
\frac{\sin(\pi b)}{\pi b}
\frac{d}{dz}\mid_{z=0}\ _3F_2(a,-a,z;1+b,1-b;1).
\end{equation}
In \S \ref{section:Comparison to Zagier}, 
we show that associator versions of the formal power series 
(\ref{Gauss-Manin generating function})
 and 
(\ref{Zagier's expression})
coincides.
\end{remark}
\section{Associator and Hopf algebroid}
\subsection{Fundamental algebroid of moduli spaces}

We recall the structure of Hopf algebroids $\Cal A_{n,dR},\Cal A_{n,B}$ 
of the moduli space $\Cal M_n=\Cal M_{0,n}$ 
of $n$-punctured genus zero curves in this subsection.
\begin{definition}
We define the set of tangential points $T_n$ of $n$ points in genus zero curve
as the set of planer trivalent tree with $n$ terminals.
For example 
$$
T_4=\{\overline{01},\overline{10},\overline{0\infty},\overline{\infty 0},
\overline{1\infty},\overline{\infty 1}\}
$$ 
Thus $\# T_4=3\times 2$, $\# T_5=15\times 4$, etc.
\end{definition}
Then we can define the pro-nilpotent algebroid $\Cal A_{n,dR},\Cal A_{n,B}$
over the set $T_n$ as follows.
\begin{definition}
For two points $a,b \in T_n$,
the bifiber of the algebroid $\Cal A_{n,dR}=\{\Cal A_{n,dR,ab}\}_{ab}$ is defined as
the following generators and reltaions. 
\begin{enumerate}
\item
(Genrators)
$t_{ij}$ with $1\leq i<j \leq n$.
We use the notation $t_{ji}=t_{ij}$ for $i<j$.
\item
(Relations)
\begin{enumerate}
\item
$[t_{ij},t_{kl}]=0$
\item
$[t_{ij},t_{ik}+t_{kj}]=0$
\item
$\sum_{j\neq i}t_{ij}=0$
\end{enumerate}
\end{enumerate}
Then $\Cal A_{n,dR}$ is the completed de Rhan fundamental group algebra
of $\Cal M_{n}$ and
has a standard coproduct $\Delta(t_{ij})=t_{ij}\otimes 1+1\otimes t_{ij}$.
\end{definition}
\begin{definition}
\begin{enumerate}
\item
Two tangential base points $a,b \in T_n$ are adjacent
if it can be transformed by elementary change $H \leftrightarrow I$.
\item
Two tangentail base points $a,b \in T_n$ are neighbours 
if it can be transformed by twisting with respect to a edge. 
\item
$\Cal A_{n,B}=\{\Cal A_{n,B,ab}\}_{ab}$ is a pro-nilpontent algebroid
generated by two type of generators:
\begin{enumerate}
\item
path $p_{ab}$ connecting two adjacent tangential base points.
\item
small circle $c_{ab}$ connecting two neibours.
\item
Relations on $\Cal A_{n,B}$ are generated by
2-cycle relations, 3-cycle relations, 5-cycle relations.
\end{enumerate}
Then the $\Cal A_{n,B}$ is the completed groupoind algebra of $\Cal M_{n}$.
\end{enumerate}
\end{definition}
\begin{definition}(Category $\Cal C$)
We define the abelian category $\Cal C$ as follows.
An object $V$ of $\Cal C$ is a triple $(V_{dR},V_{B},c_V)$ consisting of
\begin{enumerate}
\item
$\bold Q$-vector space $V_{dR}$,
\item
$\bold Q$-vector space $V_{B}$, and
\item
an isomorphism $V_{B}\otimes \bold C\simeq \bold V_{dR}\otimes \bold C$
\end{enumerate}
Sometimes one consider profinite version. In this case, $\otimes \bold C$
means the completed tensor product.
Morphism form $f:V\to W$ is a pair of morphisms $f_{dR}:V_{dR}\to W_{dR}$ 
and $f_{B}:V_{B}\to W_B$ compatible with the comparison maps.
The category $\Cal C$ becomes a tensor category by tensoring each
$dR$ and $B$ components
\end{definition}
\begin{definition}
We define the category $M^{inf}$ be the category whose objects are
$\Cal M_{n}$ and morphisms are
generated by 
infinitesimal inclusions.
\end{definition}
\begin{definition}
We can define two functors $\Cal A_{dR}, \Cal A_{B}
:M^{inf} \to Hopf_{\bold Q}$ from $M^{inf}$ to the category
of Hopf algebroids by attaching
de Rham fundamental groups and Betti fundamental groups.
\end{definition}

\subsection{Choice of coordinate}
Let $C$ be a genus zero curve and
$P=(C,p_1, \dots, p_n)$ ($p_i \in C$) an element in $\Cal M_{n}$.
We choose a coordinate $t$ of $C$
such that $t(p_{n-2})=0,t(p_{n-1})=0,t(p_{n})=0$.
Using the coordinate $t$, $\Cal M_{n}$ is identified with an open set
of $\bold A^{n-3}$ defined by 
$$
\{(x_1, \dots, x_{n-3})\mid\  x_i\neq x_j \text{ for }i\neq j,
x_i\neq 0,1 \text{ for all }i\}
$$ 
by setting $x_k=t(p_{k})$.
This coordinate is called the distinguished coordinate.
By taking the distinguished coordinate of $\Cal M_{4}$,
the underlying curve is identified with $\bold P^{1}-\{0,1,\infty\}$.
\begin{definition}[admissible function, admissible differential form]
\begin{enumerate}
\item
Let $S=(i,j,k,l)$ be a ordered subset of distinct elements in $[1,n]$.
For an element $P=(C, p_1, \dots, p_n)$ be an element of $\Cal M_{n}$.
There is a unique coordinate $t$ of $C$ such that 
$t(p_i)=0,t(p_j)=1,t(p_k)=\infty$. The value $t(p_l)$
at $p_l$ gives rise to an algebraic function on $\Cal M_n$,
which is denoted by $\varphi_S$.
The set of admissible functions is denoted by $Ad(\Cal M_n)$.
\item
Let $x_1, \dots, x_n-3$ be the distinguished coordinate.
An element in the linear span of
$\frac{dx_i}{x_i},\frac{dx_i}{x_i-1},\frac{d(x_i-x_j)}{x_i-x_j}$
is called an admissible differential form.
\end{enumerate}
\end{definition}
\begin{remark}
\begin{enumerate}
\item
$\varphi \in Ad(\Cal M_n)$ defines a morphism $\Cal M_n\to \Cal M_{4}$.
and a morphism of algebroids $\Cal A_n \to \Cal A_4$.
\item
If $S\cap \{n-2,n-1,n\}=\emptyset$, using the distinguised coordinates
of $\Cal M_n$, we have
$$
\varphi_S(P)=\frac{(x_l-x_i)(x_j-x_k)}{(x_l-x_k)(x_j-x_i)}.
$$
Therefore $\varphi_S$ is invariant under substitutions
$i \leftrightarrow l, 
j \leftrightarrow k$ and 
$i \leftrightarrow j, 
k \leftrightarrow l$.
\item
The following functions are admissible functions.
\begin{align*}
\frac{x_i}{x_j}=&\frac{(x_i-0)(x_j-\infty)}{(x_j-0)(x_i-\infty)},\quad
1-\frac{x_i}{x_j}=\frac{(x_j-x_i)(\infty-0)}{(x_j-0)(\infty-x_j)}. \\
1-x_i=&\frac{(1-x_i)(\infty-0)}{(1-0)(\infty-x_i)}
\end{align*}
\end{enumerate}
\end{remark}
\begin{proposition}
The set of functorial isomophisms from $\Cal A_{B}\otimes \bold C$ 
to
$\Cal A_{dR}\otimes \bold C$ sending small half circle $\log(c_{ij})$ 
to $\pi \bold i t_{ij}$
is identified with the set of
assoicators. The one to one correspondence is given by
$$
\Cal A_{4,B,\overline{01},\overline{10}}\ni [0,1]\mapsto
\Phi \in \Cal A_{dR,4}=\bold C\<\<e_0,e_1\>\>
$$
Here $e_0$ and $e_1$ are the dual basis of 
$\displaystyle\omega_0=\frac{dx}{x}$
and $\displaystyle\omega_1=\frac{dx}{x-1}$, respectively. 
\end{proposition}
By the above proposition, 
we have an isomorphism of Hopf algebra
$$
c_{\Phi,n}:\Cal A_{n,B}\otimes \bold C \xrightarrow{\simeq} \Cal A_{n,dR}\otimes \bold C.
$$
associated to a given assoicator $\Phi$.
This isomorphism gives an object $\Cal A_n^{\Phi}=(\Cal A_{n,dR},\Cal A_{n,B},c_{\Phi,n})$.
The isomorphism $c_{\Phi,n}$ is called the $\Phi$-comparison map.
\begin{proposition}
\begin{enumerate}
\item
Let $3\leq m<n$ be integers and
$f$ morphsim defined by
$$
f:\Cal M_{n}\to \Cal M_{m}:(x_1,\dots, x_{n-3}) \to
(x_1,\dots, x_{m-3})
$$
Then for $\star=dR, B$, the induced maps of algebroids
$$
\Cal A_{n,\star} \to \Cal A_{m,\star}
$$
are compatible with the $\Phi$-comparison maps.
\item
Let $3\leq m,n$ be integers. Then
a morphsim
$$
f:\Cal M_{n+m-3}\to 
\Cal M_{n}\times \Cal M_{m}
$$
$$
(x_1,\dots,x_{n-3},y_1, \dots, y_{m-3})
\mapsto
(x_1, \dots, x_{n-3})\times (y_1, \dots, y_{m-3})
$$
induces a morphism of algebroids 
$$
f:\Cal A_{n+m-3}\to \Cal A_{n} \otimes \Cal A_{m}
$$
in $\Cal C$.
\item
Let $3\leq m<n_1,n_2$ be integers. Then
the natural morphsim
$$
f:\Cal M_{n_1}\times_{\Cal M_{m}}\Cal M_{n_2}\to 
\Cal M_{n_1}\times \Cal M_{n_2}
$$
induces a morphism of algebroids in $\Cal C$.

\end{enumerate}
\end{proposition}
The coefficient $c_{\Phi,I}$ of $e_{i_1}e_{i_2}\dots e_{i_k}$
in $c_{\Phi,4}([0,1])$ 
is written as $\int_{[0,1]}^{\Phi}\omega_{i_1}\dots \omega_{i_k}$.
We define $\Phi$-multiple zeta value similarly.
A $\Phi$-multiple zeta value is written as
$$
\zeta_{\Phi}(m_1, \dots, m_k)=
\int_{[0,1]}^{\Phi}\omega_{0}^{m_k-1}\omega_{1}\dots \omega_{0}^{m_1-1}\omega_{1} 
$$
It is a coefficint of the associator $\Phi$.

\subsection{$\Cal A$-module}
Let $T$ be a set and $\Cal A$ a Hopf algbroid object in $\Cal C$
over $T$.
We define the notion of $\Cal A$-module.
\begin{definition}
Let $M=(M_a)_{a\in T}=(M_{dR,a},M_{B,a},c_{M,a})_{a\in T}$ 
be an object in $\Cal C$
indexed by $a\in T$.
$M$ is called an $\Cal A$-module 
if it is equipped with 
an action of 
$\Cal A$ in $\Cal C$ 
$$
\mu_{M}:\Cal A \otimes M \to M
$$
which is associative and unitary.
Here action of algebroid is given by a morphism
$$
\Cal A_{ab}\otimes M_a \to M_b.
$$
in $\Cal C$.
\end{definition}
\begin{remark}
Let $M, N$ be $\Cal A$ module. Then using coproduct structure of
$\Cal A$, $M\otimes N$ is equipped with $\Cal A$ module. 
\end{remark}
\begin{example}
\begin{enumerate}
\item
Let $4\leq m<n$ and 
$f:\Cal M_{n}\to \Cal M_{m}$ be the map defined by
$(x_1, \dots, x_{n-3})\mapsto (x_1, \cdots, x_{m-3})$.
Then we have an algebroid homomorphism
$f:\Cal A^{\Phi}_{n}\to \Cal A^{\Phi}_m$.
Therefore for a fixed $p\in T_m$, by setting 
$M_a=\Cal A_{m,p,f(a)}^{\Phi}$ we have an $\Cal A_n^{\Phi}$-module.
It is called a pull back of the map $f$.
\item
By taking an abelianization $\Cal A^{\Phi,ab}_n$ of $\Cal A^{\Phi}_n$, 
we have a homomorphism of Hopf algebroids
$$
\Cal A_n^{\Phi} \to \Cal A^{\Phi,ab}_n.
$$
By choosing a base point $p\in T_n$, we have have an $\Cal A^{\Phi}_n$-module
$\Cal A^{\Phi,ab}_{n,p*}$.
In particular, by using the distinguished
coordinate $x$, $\Cal A_4$ module $x^{\alpha}\bold Q[[a]]$
is defined by taking the base point as $\overline{01}$,
\item
Let $\varphi$ be an admissible function on $\Cal M_n$
and $\alpha$ formal parameter.
The morphism $\Cal M_n \to \Cal M_4$ induced by $\varphi$ is
also denoted by $\varphi$ and $x$ be the distinguished coordinate
of $\Cal A_4$.
We define $\Cal A_{n}[[\alpha]]$-module
$$
\varphi^{\alpha}\bold Q[[\alpha]]
$$
by the pull back $\varphi^*(x^{\alpha}\bold Q[[a]])$ of 
$x^{\alpha}\bold Q[[a]]$.
We define
$$
\bigg(\prod_{i=1}^m\varphi^{\alpha_i}\bigg)\bold Q[[\alpha_1,\dots,\alpha_m]]=
\varphi_1^{\alpha_1}\bold Q[[\alpha_1]]\widehat{\otimes}
\cdots \widehat{\otimes}
\varphi_m^{\alpha_m}\bold Q[[\alpha_m]]
$$
\end{enumerate}
\end{example}
\begin{proposition}
Let $\varphi_i$, $(i=1, \dots, m)$, $\psi_j$, $(j=1, \dots, l)$
be admissible functions on $\Cal M_n$ and $a_{ij}\in \bold Z$. 
We assume that $\psi_j=\prod_{i=1}^m\varphi^{a_{ij}}$
We set
$$
L_j=\sum_i^{m}a_{ij}\alpha_{i}
$$
for $j=1,\dots, l$.
Then
$$
\bigg(\prod_{i=1}^m\varphi_i^{\alpha_i}\bigg)\bold Q[[\alpha_i]]
=\bigg(\prod_{j=1}^l\psi_j^{L_j}\bigg)\bold Q[[\alpha_i]].
$$
as $\Cal A_n^{\Phi}$ module.
\end{proposition}
Let $\Cal A$ be an algebraoid in $\Cal C$ and $M$ be 
an $\Cal A$-module. 
We define the dual $M^*$ of $M$ using antipodal.

\begin{proposition}[Descent theory for $\Cal A^{\Phi}_4$-module.]
Let $M$ be an $\Cal A^{\Phi}_4$-module.
Assume that $M_{dR}$ is constant, i.e. the map
$$
M_{dR}\xrightarrow{e_0,e_1} M_{dR}\oplus M_{dR}
$$
is the zero map. Then $M$ is the pull back of an object
$N$ in $\Cal C$ such that $M=\pi^*N$, where 
$\pi:\Cal A_{4}^{\Phi}\to \bold Q$ is the augmentation map.
\end{proposition}
\begin{proof}
Let $I=\ker(A^{\Phi}_4 \to \bold Q)$ be the augmentation ideal.
Then we have the following commutative diagram whose vertical
arrows come from comparison maps and are isomorphisms.
$$
\begin{matrix}
I_{B,ab}\otimes M_{B,a}\otimes \bold C & \xrightarrow{\alpha} 
& M_{B,b}\otimes \bold C \\
\downarrow & & \downarrow \\
I_{dR,ab}\otimes M_{dR,a}\otimes \bold C & \xrightarrow{\beta}
 & M_{dR,b}\otimes \bold C \\
\end{matrix}
$$
Since $\beta$ is the zero map, $\alpha$ is the zero map.
Therefore $M$ is induced from an object in $\Cal C$.
\end{proof}

\subsection{Comparison map and actions}

\subsubsection{de Rham framing}
Let $M$ be an $\Cal A_4^{\Phi}[[\alpha_i]]$-module.
and $c_M:M_{B}\to M_{dR}$ be the comparison map of $M$.
\begin{definition}
\begin{enumerate}
\item
Let $y \in T_4$. A de Rham framing of $M$
is a pair of 
homomorphisms 
$\alpha:\bold Q[[\alpha_i]] \to M_{B,y}$ and
$\beta:M_{dR} \to \bold Q[[\alpha_i]]$
of $\bold Q[[\alpha_i]]$-modules.
\item
Let $f=(\alpha,\beta)$ be a framing of $M$ at $y$,
and $\gamma$ be an element in $\Cal A_{4,B,yz}^{\Phi}$
the value $f(\gamma)$ of $f$ at $\gamma$ is defined by
$$
\beta\circ c_M\circ \gamma\circ \alpha \in \bold Q[[\alpha_i]].
$$
\end{enumerate}
\end{definition}

Let $f=(\alpha,\beta)$ be a framing of $M$ at $\overline{01}$.
The $dR$-part $M_{dR}$ of $M$ is a $\Cal A_{4,dR}^{\Phi}
\simeq \bold C\<\<e_0,e_1\>\>$ module.
Let $E_0, E_1$ be actions 
of $e_0$ and $e_1$ on $M_{dR}$.
The action of $\varphi=\varphi(e_0,e_1)\in \bold C\<\<e_0,e_1\>\>$
on $M_{dR}$ is denoted by $\varphi(E_0,E_1)$.
Since the actions of 
$\Cal A^{\Phi}_{4,B}$ and
$\Cal A^{\Phi}_{4,dR}$ on $M_{B}$ and $M_{dR}$
are
compatible via the comparison map, using the associator $\Phi$, we have
$$
f([0,1])=\beta 
c_{M}[0,1] \alpha =
\beta c_{\Cal A^{\Phi}_4}([0,1])c_M\alpha
=\beta \Phi(E_0,E_1)c_M\alpha\in \bold Q[[\alpha_i]].
$$

\subsubsection{Example 1}

We consider a module $M_{dR}=\bold Q[[a]]^{\oplus 2}$.
Let $P_0,P_1$ be endomorphisms defined as
(\ref{connection matrix 2xn}).
Therefore it defines a $\Cal A_{4,dR}$ module structure on $M_{dR}$.
The action of $\varphi$ of (\ref{typical element})
 is given by (\ref{2,..2 generating matrix}).

\subsubsection{Example 2}
We consider a module $M_{dR}=\bold Q[[a,b]]^{\oplus 2}$.
Let $P_0,P_1$ be endomorphisms defined by
\begin{equation}
\label{connection inducing beta function}
P_0=\left(
\begin{matrix}
0 & 0 \\
0 & -c
\end{matrix}\right),
P_1=
\left(\begin{matrix}
0 & 0 \\
-b & c-b
\end{matrix}
\right).
\end{equation}
of $M_{dR}$.
Then it defines a $\Cal A_{4,dR}$ module structure on $M_{dR}$.
For $I=(i_1,\dots, i_n) \in\{0,1\}^n$, we have
$$
(0,1)P_I
\left(
\begin{matrix}
1 \\ 0
\end{matrix}
\right)=
\begin{cases}
0 \text{ if }i_n=0 \\
(-b)(-c)^p(c-b)^q \text{ if }i_n=1 \text{ where } 
p=\#\{i_k=0\}-1,q=\#\{i_k=1\},\\
\end{cases}
$$
\begin{proposition}
\label{abelianization of end with index 1}
Let 
$$
\varphi(e_0,e_1)=1+\varphi_0(e_0,e_1)e_0+\varphi_1(e_0,e_1)e_1
$$ 
be an elenemt $\bold C\<\<e_0,e_1\>\>$.
Then we have
$$
(0,1)\varphi(P_0,P_1)
\left(
\begin{matrix}
1 \\ 0
\end{matrix}
\right)=(-b)\varphi_1^{ab}(-c,c-b) \in \bold C[[b,c]].
$$
where $\varphi_0^{ab}(-b,c-b)$ is the 
image under the ablianization map
$\bold C\<\<e_0,e_1\>\>\to\bold C[[b,c]]$.
\end{proposition}

\section{Higher direct images for $\Cal A^{\Phi}$-modules}
In this section, we define relative cohomologies and study their
properties.

\subsection{Relative cohomology}
In this section, $\Cal A_n$ is $\Cal A_n^{\Phi}, \Cal A_{n,dR}^{\Phi}$ or
$\Cal A_{n,B}^{\Phi}$.
Let $4\leq m<n$ and $M$ be an $\Cal A_{n}$-module.
Let $f:\Cal M_{n}\to \Cal M_{m}$ be a map defined by
$$
(x_1, \dots, x_{n-3}) \mapsto (x_1, \dots, x_{m-3})
$$
and $f:\Cal A_{n}\to \Cal A_{m}$ be the induced
morphism of algebroid objects in $\Cal C$.
We define a complex $F(\Cal A_n)$ by
$$
\begin{matrix}
\dots \to & \Cal A_n\otimes \Cal A_n \otimes\Cal A_n
& \to & \Cal A_n\otimes \Cal A_n & \to  0 \\
& 
x\otimes y\otimes z &\mapsto&
xy\otimes z 
-x\otimes yz 
\end{matrix}
$$
Then the map $\Cal A_n\otimes \Cal A_n \to \Cal A_n$
defined by $x\otimes y \mapsto xy$ defines a free
$(\Cal A_n\otimes \Cal A_n^0)$ resolution
$F(\Cal A_n) \to \Cal A_n$. Therefore
$F(\Cal A_n)\otimes_{\Cal A_n} \Cal A_m$
is a free $\Cal A_n$-resolution of $\Cal A_m$.
We note the relation
$$
Hom_{\pi_1(\Cal M_n)}(\Cal A_{m,B},M_B)=M_B^{N},
$$
whrer $N=\ker(\pi_1(\Cal M_n)\to \pi_1(\Cal M_m))$.
Motivated by the above relation, we define 
$\bold R f_*M$ by $Hom_{\Cal A_n}(F(\Cal A_n)\otimes_{\Cal A_n} \Cal A_m,\Cal M) $.
More concretely, we have
\begin{align}
\label{definition of higher direct image}
\bold R f_*M:
Hom(\Cal A_m,\Cal M) &\xrightarrow{d^0}
Hom(\Cal A_n\otimes \Cal A_m,\Cal M) \\
\nonumber
&\xrightarrow{d^1}
Hom(\Cal A_n\otimes \Cal A_n\otimes \Cal A_m,\Cal M)\xrightarrow{d^2} \dots.
\end{align}
Here $d^0$ is given by $d^0(\varphi)(x\otimes y)=x\varphi(y)-\varphi(f(x)y)$.
The right $\Cal A_m$ acion on
$F(\Cal A_n)\otimes_{\Cal A_n} \Cal A_m$
induces a left $\Cal A_{m}$-module structure on $\bold R f_*M$.
As a consequence, we have a left $\Cal A_m$ module
$\bold R^if_*M=H^i(\bold R f_*M)$.
If $M=(M_{dR},M_{B},c_M)$ is an $\Cal A_n^{\Phi}$-module, then
$$
(\bold R^if_*M)_{dR}=
\bold R^if_*(M_{dR}),\quad
(\bold R^if_*M)_{B}=
\bold R^if_*(M_{B})
$$
If $f:\Cal M_{n}\to pt=\Cal M_{3}$,
$\bold R^if_*M$ is denoted by $H^i_{\Phi}(\Cal M_{n},M)$

\subsection{Hochschild-Serre-Leray spactral sequence}
Let $4\leq l<m<n$ be natural numbers and
$\Cal M_{n} \xrightarrow{g} \Cal M_{m}\xrightarrow{f}\Cal M_{l}$
be a map defined by 
$(x_1, \dots, x_{n-3})\mapsto(x_1, \dots, x_{m-3})\mapsto(x_1, \dots, x_{l-3})$
\begin{proposition}
The homomorphism
$$
F(\Cal A_n)\otimes_{\Cal A_n}(F(\Cal A_m)\otimes_{\Cal A_m}\Cal A_l)
\to F(\Cal A_n)\otimes_{\Cal A_n}\Cal A_l$$
induces a quasi-isomorphism
$$
\bold R(fg)_*M 
\xrightarrow{\sim}
\bold Rf_*(\bold Rg_*M).
$$
\end{proposition}
\subsection{Fundamental algebroid of fibers and higher direct image}
\subsubsection{Fibers of higher direct images}
We give a method to compute the higher direct image for 
$f:\Cal M_{n} \to 
\Cal M_{m}$ for $dR$ and $B$. 
Let $f:T_n\to T_m$ be the  corresopnding map for infinitesimal points,
and $y$ an element of $T_m$.
We set $T_{n,m}(y)=f^{-1}(y)$.
\begin{definition}
Let $\Cal A_{n,m,B,y}$ (resp. $\Cal A_{n,m,dR,y}$) be
the subalgebroid of $\Cal A_{n,B}$, (resp. $\Cal A_{n,dR}$)
generated by the images of $\Cal A_{4,B}$
induced by infinitesimal inclusions of $\Cal M_{4}\to \Cal M_{n}$ contained in
the fiber of $y$.
Then the image of $\Cal A_{n,m,B,y}\otimes \bold C$ is equal to
 $\Cal A_{n,m,dR,y}\otimes \bold C$.
Therefore $\Cal A_{n,m,B,y}$ and $\Cal A_{n,m,dR,y}$
defines a Hopf algebroid object in $\Cal C$ on $T_{n,m}(y)$,
which is denoted by $\Cal A_{n,m,y}$
For $x\in T_{n,m}(y)$, $\Cal A_{n,m,y,x}$ is denoted by
$\Cal A_{n,m,x}$.
\end{definition}
\begin{remark}
The $B$-part $\Cal A_{n,m,B}$ can be interpreted as follows.
Let $N_{n,m}$ be the kernel of $\pi_1^B(\Cal M_{n}) \to \pi_1^B(\Cal M_{m})$.
Then $N_{n,m}$ becomes a fibered groupoid over the map $T_n\to T_m$.
We can easily see that $\Cal A_{n,m}$ is the nilpotent completion of $N_{n,m}$.
\end{remark}
\begin{proposition}
We choose $x\in T_n,y\in T_m$ such that $f(x)=y$. 
We have the following exact sequence:
$$
0 \ot \Cal A_{m,y} \ot \Cal A_{n,x} 
\xleftarrow{d_0} \Cal A_{n,x}\otimes  \Cal A_{n,m,x}
\xleftarrow{d_1}\ \Cal A_{n,x}\otimes  \Cal A_{n,m,x}\otimes  \Cal A_{n,m,x}
\ot \cdots
$$
Here 
$d_0(x\otimes y)=xy -x\epsilon(y)$,
$d_1(x\otimes y\otimes z)=xy\otimes z -x\otimes yz +x\otimes y\epsilon(z)$,
\dots, where $\epsilon:\Cal A_{n,m,x}\to \bold Q$ is the augmentation. 
This becomes a  free $\Cal A_{n,x}$ resolution of $\Cal A_{m,y}$.
\end{proposition}
\begin{proof}
We reduce the proposition to the $B$-part.
Let $f:G\to H$ be a surjective homomorphism of group and $N$ be the kernel of $f$.
We prove that the sequence
\begin{equation}
\label{compare relative and fiber}
0 \ot \bold Q[H] \ot \bold Q[G] 
\xleftarrow{d_0} \bold Q[G\times N]
\xleftarrow{d_1}\ \bold Q[G\times N^2]
\ot \cdots
\end{equation}
is exact. We choose a set theoretic section $s:H\to G$.
Then
\begin{align*}
&\theta_0:\bold Q[H]\to \bold Q[G]:h\to s(h) \\
&\theta_1:\bold Q[G]\to \bold Q[G\times N]:g\to g\otimes g^{-1}s(g) \\
&\theta_2:\bold Q[G\times N]\to \bold Q[G\times N^2]:
g\otimes n\to g\otimes n\otimes n^{-1}g^{-1}s(ng) \\
&\dots
\end{align*}
gives a null homotopy. 
Therefore the sequence (\ref{compare relative and fiber})
 is an exact sequence.
By taking a nilpotent completion, we have the proposition for the $B$-part.
\end{proof}

\begin{corollary}
The complex
$\bold Rf_*M_{y}$ is quasi-isomorphic to the complex
$Hom_{\Cal A_{n,m,x}}(F(\Cal A_{n,m,x})
\otimes_{\Cal A_{n,m,x}}\bold Q,M_{x}).$
For the $B$-part, the action of 
$\Cal A_m$ on $\bold Rf_*M_{B}$ is given by the monodromy action.
\end{corollary}

\subsection{Comparison to de Rham cohomologies and chain complexes}
\subsubsection{Comparison to de Rham complexes}
We show that the $B$-part is equal to Gauss-Manin connection
with the coefficient in $M_{dR}$. 
If $m=n-1$, then using the commutation relation, $\Cal A_{n,dR}$ can be written as
the formal power series ring.
$$
\Cal A_{n,dR}=\Cal A_{n-1,dR}\<\<t_{n,1},\dots, t_{n,n-2}\>\>
$$
as a vector space. The multiplication rule is given by the 
commutation relation.
Let $M_{dR}$ be a continous $\Cal A_{n,dR}$-module. Then the action of
$t_{ij}$ gives
a nilpotent endomorphism $E_{ij}$
on $M_{dR}$.
\begin{proposition}
As a vector space $\bold Rf_*M_{dR}$ is quasi-isomorphic to
$$
\bold Rf_*'M_{dR}:M_{dR}
\xrightarrow{\nabla} M_{dR}\otimes \Omega^1_{n/m}
\xrightarrow{\nabla} M_{dR}\otimes \Omega^2_{n/m} \to\dots
$$
$$
\bold Rf_*''M_{dR}:M_{dR}
\xrightarrow{\nabla} M_{dR}\otimes \Omega^{1}_{\Cal M_{0,n}/\Cal M_{0,m}}
\xrightarrow{\nabla} M_{dR}\otimes \Omega^{2}_{\Cal M_{0,n}/\Cal M_{0,m}}
 \to\dots
$$
Here $\Omega^{\bullet}_{n/m}$ is a subcomplex of the relative de Rham complex
$\Omega^{\bullet}_{\Cal M_{0,n}/\Cal M_{0,m}}$ generated by
$\displaystyle\frac{d(x_i-x_j)}{x_i-x_j}$.
As a consequence, 
\begin{enumerate}
\item
if $M$ is finite dimensional, then
$\bold R^if_*M_{dR}$ is also finite dimensional, and
\item
$\bold R^if_*M_{dR}=0$ if $i>n-m$.
\end{enumerate}
\end{proposition}
\begin{proof}
Since the action of $\<E_{ij}\>$ are nilpotent, 
we can show that $\bold Rf_*'M_{dR}$ and
$\bold Rf_*''M_{dR}$ are quasi-isomorphic
by the induction of the length of nilpotent filtrations.
\end{proof}
To give an explicit quasi-isomorphism, it is convenient to introduce
the bar complex.
Let $\overline{B_{n}}$ be the reduced bar complex of logarithmic bar complex 
$\Omega_n^{\bullet}$. 
Then the topological dual of $\Cal A_n$ is isomorphic
$B_n=H^0(\overline{B_n})\subset\overline{B_n}$ and $H^i(B_n)=0$ for $i\neq 0$.
Then $B_n$ becomes a Hopf algebra and 
the $\Cal A_n$ action on $M_{dR}$ yields a right $B_n$-comodule structure on
$M_{dR}$.
By the definition
(\ref{definition of higher direct image}),
 $\bold Rf_*M_{dR}$ is equal to
$$
0\xrightarrow{d_0} M_{dR}\otimes B_m 
\xrightarrow{d_1} M_{dR}\otimes B_n\otimes B_m \to
M_{dR}\otimes B_n\otimes B_n\otimes B_m \to \dots.
$$
For example $d_0,d_1$ is given by the formula
\begin{align*}
d_0(a\otimes  m)&=\Delta_m(a)\otimes m
-a\otimes  \Delta_M(m) \\
d_1(a\otimes b\otimes m)&=\Delta_m(a)\otimes b \otimes m-a\otimes \Delta(b)\otimes m
+a\otimes b\otimes \Delta_M(m) 
\end{align*}
Here $\Delta_m:B_m \to B_m \otimes B_n$,
$\Delta:B_n \to B_n \otimes B_n$ and 
$\Delta_M:M_{dR} \to B_n \otimes M_{dR}$
are the coproducts.
\begin{proposition}
\begin{enumerate}
\item
Let $\psi^k:M_{dR}\otimes B_n^{\otimes k}\otimes B_m 
\to M_{dR}\otimes \Omega_{n/m}$
be a map defined by
$$
m\otimes a_1\otimes \cdots \otimes a_k \otimes b \mapsto
m\otimes \pi(a_1)\cdots \pi_k(a_1)\epsilon(b),
$$
where $\epsilon:B_m\to \bold Q$ is the augmentation.
Then $\sum_k\psi^k$ is a homomorphism of complex and 
quasi-isomorphism.
\item
The action of $\Cal A_{m,dR}$ on $\bold R^if_*M_{dR}$
is equal to Gauss-Manin connection. 
\end{enumerate}
\end{proposition}

\subsubsection{Comparison to chain complex}
Let $M$ be an $\Cal A^{\Phi}_{n}[[\alpha_i]]$ module
and $M^*=Hom_{\bold Q[[\alpha_i]]}(M,\bold Q[[\alpha_i]]^{\Phi})$.
Then $M_B$ and $M_B^*$ define local systems on $\Cal M_n$.
The homology and the cohomology in the coefficient in 
$M_{B}^*$
and $M_{B}$ is denoted by
$
H_i^B(\Cal M_n,M_B^*)$ and $H^i_B(\Cal M_n,M_B)$,
respecitvely.
We have the natural pairing
$$
H_i^B(\Cal M_n,M_B^*)\otimes H^i_B(\Cal M_n,M_B) \to \bold C[[\alpha_i]]
$$
and via this map we have the following evaluation homomorphism
$$
ev:H_i^B(\Cal M_n,M_B^*)\to H^i_B(\Cal M_n,M_B)^*.
$$
\subsection{$\Phi$-integral}
We have an isomorphism
$$
H^i(\Cal M_n, M)_{dR}\simeq
H^i_{dR}(\Cal M_n, M_{dR})
$$
and
$$
H^i(\Cal M_n, M)_{B}\simeq
H^i_B(\Cal M_n,M_B).
$$
The homology $H_i^B(\Cal M_n,M^*)$ is identified with the homology group
of chains complex with the coefficient in $M^*$.
An element $\sigma$ of the chain complex 
is a linear combination of $[\gamma,f]$ where
$\gamma$ is an $i$-chain in $\Cal M_n$ and $f$ is a section of $M^*$
on $\gamma.$
\begin{definition}[$\Phi$-integral, twisted chain]
\begin{enumerate}
\item
Let $\sigma=[\gamma,f] \in 
H_i^B(\Cal M_n,M_B^*)$
and $\omega\in H^i_{dR}(\Cal M_n, M_{dR})$.
We define a $\Phi$-integral by
$$
\int^{\Phi}_{\gamma}f\omega=
ev(\sigma)(c_H^{-1}(\omega))\in \bold C[[\alpha_i]]
$$
$\Phi$-integral defines a pairing
$$
H_i^B(\Cal M_n,M_B^*) \otimes H^i_{dR}(\Cal M_n, M_{dR})
\to \bold C[[\alpha_i]]
$$
\item
Let $\varphi_i$ ($i=1, \dots, l$) be admissible functions on $\Cal M_n$,
$D$ a domain defined by $0 \leq x_1\leq x_2 \leq \cdots \leq x_{n-3}\leq 1$
for some distinguished coordinates $x_1, \dots, x_{n-3}$.
Asuume that the values of $\varphi_i$ are positive and real on $D$.
The twisted chain on $D$ with the product of positive real
branchs of $\varphi_i^{\alpha_i}$ 
is denoted by
$\prod_{i=1}^l\varphi^{\alpha_i}_{D}$.
\end{enumerate}
\end{definition}

\subsection{Regularization of cycles and specialization of exponents}
$M$ be an $\Cal A_4^{\Phi}$-module and
$M^*$ its dual.
Then 
\begin{align*}
&Mx^{\alpha}(1-x)^{\beta}=
M\otimes x^{\alpha}(1-x)^{\beta}\bold Q((\beta))[[\alpha]] \\
&M^*x^{-\alpha}(1-x)^{-\beta}=
M\otimes x^{-\alpha}(1-x)^{-\beta}\bold Q((\beta))[[\alpha]]
\end{align*}
becomes an $\Cal A_4^{\Phi}((\beta))[[\alpha]]$-module.
Let $v\in M_{B,\overline{01}}^*$ 
be an element invariant under
the action of local monodromy at $\overline{01}$
and $w \in M_{dR}$.
By the standard arguement for the regularization of topological cycles
around $0$ and $1$, we have the following proposition.
\begin{proposition}
The extension $\tilde v(\alpha)$ to 
$[0,1]$ of the element $v \alpha x^{\alpha}(1-x)^{\beta}$
defines a homology class in $H_1(\Cal M_4, M^*x^{\alpha}(1-x)^{\beta})$.
\end{proposition}
\begin{proof}
The action of the local monodromy $\rho_0$ at $0$ on $vx^{\alpha}(1-x)^{\beta}$
is given by
$$
vx^{\alpha}(1-x)^{\beta} \to \bold e(\alpha) vx^{\alpha}(1-x)^{\beta}.
$$
Therefore
$$
\alpha vx^{\alpha}(1-x)^{\beta}
=\frac{\alpha}{\bold e(\alpha)-1}(\rho_0-1)vx^{\alpha}(1-x)^{\beta}
=\frac{\alpha}{\bold e(\alpha)-1}\partial(vx^{\alpha}(1-x)^{\beta}_{c_0})
$$
where $c_0$ is the small circle around $0$.
The action of $\rho_1-1$ is invertible where $\rho_1$ is the local monodromy.
Therefore $v \alpha x^{\alpha}(1-x)^{\beta}$
defines a homology class in $H_1(\Cal M_4, M^*x^{\alpha}(1-x)^{\beta})$.
\end{proof}
\begin{definition}
The extension of $\alpha vx^{\alpha}(1-x)^\beta$ is called the regularized cycle.
\end{definition}
As a consequence of the above proposition, we have a pairing
$$
F(\alpha)=(\tilde v(\alpha),\omega\frac{dx}{x})\in \bold C((\beta))
[[\alpha]].
$$
The specialization $\tilde v(0)$ of $\tilde v(\alpha)$ defines an element in
$H_1(\Cal M_4,M^*(1-x)^\beta)$.
The specialization $F(0)\in \bold Q((\beta,\gamma))$ is equal to
the pairing $(\tilde v(0),\omega\frac{dx}{x})$ 
between 
$H_1(\Cal M_4,M^*(1-x)^\beta)$ and 
$H^1_{dR}(\Cal M_4,M(1-x)^{-\beta})$.
Since the monodromy action on $\tilde v(0)$ around $0$ is trivial, we have
the following proposition. 
\begin{proposition}
\label{limit tends to delta function}
$\tilde v(0)$ is contained in the image of the map
$$
H_1(\Delta_0^*,M^*(1-x)^{\beta}) \to
H_1(\Cal M_4,M^*(1-x)^{\beta}),
$$
where $\Delta_0^*$ is the small disc around $0$.
As a consequence, we have
$$
(\tilde v(0),\omega\frac{dx}{x})=(v,\omega)
$$
\end{proposition}
\begin{proof}
$\tilde v(\alpha)$ is sum of $\alpha vx^{\alpha}(1-x)^{\beta}$
and
$\frac{\alpha}{\bold e(\alpha)-1}vx^{\alpha}(1-x)^{\beta}_{c_0}$.
By taking the limit for $\alpha\to 0$, it is tends to
$vx^{\alpha}(1-x)^{\beta}_{c_0}$.
\end{proof}
\begin{remark}
In classical case, the function $\alpha x^{\alpha-1}$
on $[0,1]$ tends to the delta function supported at $0$ when $\alpha$ tends to
zero. The above proposition is reinterpretation of this fact 
using regularization of topological cycles.
\end{remark}
\begin{definition}
\begin{enumerate}
\item
For a section $v$ of $M$ on $[0,1]$, the pairing of $\omega\in H^1_{dR}
(\Cal M_4,Mx^{-\alpha}(1-x)^{-\beta})$ and 
the regularized cycle 
of $vx^{\alpha}(1-x)^{\beta}$ in $H_1^B(\Cal M_4, M^*x^{\alpha}(1-x)^{\beta})$
is denoted by
$$
\int^{\Phi}_{[0,1]}vx^{\alpha}(1-x)^\beta\omega
$$
\item
For a section $v$ of $M$ on $[0,1]$,
the raltive cycle modulo $\{\overline{01} \cup \overline{10}\}$ defined by $v$
is denoted by $v_{(0,1)}$.
For an element
$$
\omega\in H^1_{dR}
(\Cal M_4,Mx^{-\alpha}(1-x)^{-\beta} (\text{mod }
\overline{01}\cup \overline{10}))
$$
the paring of $v_{(0,1)}$ and $\omega$ is dentoed by
$\displaystyle
\int^{\Phi}_{(0,1)}v\omega.
$
\end{enumerate}
\end{definition}
The following proposition is direct from the definition.
\begin{proposition}
\label{difference between limit and regularization}
Let $u_0,u_1$ be elements in $M_{dR}$. Then
\begin{align*}
&\int_{[0,1]}vx^{\alpha}(1-x)^{\beta}(u_0\frac{dx}{x}+u_1\frac{dx}{x-1})-
\int_{(0,1)}vx^{\alpha}(1-x)^{\beta}(u_0\frac{dx}{x}+u_1\frac{dx}{x-1}) \\
=&
((R_0+\alpha I_M)^{-1}
v(\overline{01}),u_0)+
((R_1+\beta I_M)^{-1}
v(\overline{10}),u_1)
\end{align*}
\end{proposition}

\section{$\Phi$-Beta module and $\Phi$-hypergoemtric modules}
In this section, we define a $\Cal A^{\Phi}_4$-module 
associated
to beta function and hypergeometric functions.
We fix an associator $\Phi$ throughout this section.
\subsection{Beta module and 1-cocycle relation}
\begin{definition}
\begin{enumerate}
\item
We set $\Cal F(\chi)=x^{\alpha}(1-x)^{\beta}\bold Q[[\alpha,\beta]]$.
We define the pre-beta module $\bold B^*_{\Phi}(\alpha,\beta)$ by
$$
\bold B^*_{\Phi}(\alpha,\beta)=
H^1_{\Phi}(\Cal M_{0,4},\Cal F(\chi)).
$$
It is a $\bold Q[[\alpha,\beta]]^{\Phi}$-module in $\Cal C$.
\item
The sub modules
$$
\bold B_{\Phi,B}(\alpha,\beta)=
2\pi \bold i\alpha\cdot\bold B^*_{\Phi,B}(\alpha,\beta)\cap
2\pi \bold i\beta\cdot\bold B^*_{\Phi,B}(\alpha,\beta)\subset
\bold B^*_{\Phi,B}(\alpha,\beta)
$$ 
$$
\bold B_{\Phi,dR}(\alpha,\beta)=
\alpha\cdot\bold B^*_{\Phi,dR}(\alpha,\beta)\cap
\beta\cdot\bold B^*_{\Phi,dR}(\alpha,\beta)\subset
\bold B^*_{\Phi,dR}(\alpha,\beta)
$$ 
defines a sub object in 
$\bold B^*_{\Phi}(\alpha,\beta)$,
which is called the 
Beta module. 
\end{enumerate}
\end{definition}
$\bold B_{\Phi}(\alpha,\beta)_B$ and
(resp. $\bold B_{\Phi}(\alpha,\beta)_{dR}$) is a free
$\bold Q[[\alpha,\beta]]_B$ (resp.
$\bold Q[[\alpha,\beta]]_{dR}$) modules of rank one
generated by $\varphi$ characterized by
$\varphi(x^{\alpha}(1-x)^{\beta}_{[0,1]})=1$ (resp. $\alpha\frac{dx}{x}$).
We define modified $\Phi$-beta function $B'_{\Phi}(\alpha,\beta)$ by
$$
c_{4}(\varphi)B'_{\Phi}(\alpha,\beta)=\alpha\frac{dx}{x}
$$
In other words,
$$
B'_{\Phi}(\alpha,\beta)=\int^{\Phi}_{[0,1]}
x^{\alpha}(1-x)^{\beta}
\alpha\frac{dx}{x}
$$
Since 
$$
x^{\alpha}(1-x)^{\beta}\big(\alpha\frac{dx}{x}+\beta\frac{dx}{1-x}\big)=0
$$
in $H^1(\Cal M_{4},\Cal F_{\overline{01}*}(\chi))$, 
we have $B_{\Phi}(\alpha,\beta)=B_{\Phi}(\beta,\alpha)$.
\begin{theorem}
\label{one cocycle relation on phi beta}
We have the following one cocycle relation for modified $\Phi$-beta functions.
$$
B'_{\Phi}(\alpha,\gamma+\beta)B'_{\Phi}(\gamma,\beta)
=B'_{\Phi}(\alpha,\gamma)B'_{\Phi}(\alpha+\gamma,\beta)
$$
As a consequence,
$$
B'_{\Phi}(\alpha,\beta)=\exp(\sum_{i\geq 2}a_n(\alpha^n+\beta^n+(-\alpha-\beta)^n))
$$
for some $a_n=a_{n,\Phi}\in \bold C$.
\end{theorem}
Before proving the theorem, we define the $\Cal A_5^{\Phi}$ module
$\Cal F$.
We define $\Cal F=(1-x)^{\alpha}y^{\beta}(x-y)^{\gamma}
\bold Q[[\alpha,\beta,\gamma]]$
by taking a coordinate $p=(x,y,0,1,\infty)$ of $\Cal M_{5}$.
We set $D=\{0\leq y \leq x \leq 1,\ x,y\in \bold R\}$.
We consider another coordinates $p=(1,\eta,0,\xi,\infty)$
of $\Cal M_{4}$. Then we have relations 
$\eta=\frac{y}{x}, \xi=\frac{1}{x}$.
We define
homomorphisms $f_1, f_2:\Cal M_{4}$ by $f_1(p)=x, f_2(p)=\eta$
and $(f_1,f_2)$ by the composite
$$
\Cal M_{5}\xrightarrow{\Delta}
\Cal M_{5}\times \Cal M_{5} \xrightarrow{f_1\times f_2}
\Cal M_{4}\times \Cal M_{4}.
$$
Thus we get a homomorphism of algebroid
$$
(f_1,f_2)_*:\Cal A_{5}^{\Phi} \to \Cal A_4^{\Phi}\otimes \Cal A_4^{\Phi}
$$
and its abelianization 
$\Cal A_{5}^{\Phi,ab} \to \Cal A_4^{\Phi,ab}\otimes \Cal A_4^{\Phi,ab}$.
Since
$$
(1-x)^{\alpha}y^{\beta}(x-y)^{\gamma}=
(1-x)^{\alpha}x^{\beta+\gamma}\eta^{\beta}(1-\eta)^{\gamma}
$$
we have
\begin{align*}
\Cal F=
&(f_1,f_2)^*((1-x)^{\alpha}x^{\beta+\gamma}
\eta^{\beta}(1-\eta)^{\gamma}
\bold Q[[\alpha,\beta,\gamma]]) \\
=&(f_1,f_2)^*(
(1-x)^{\alpha}x^{\beta+\gamma}
\bold Q[[\alpha,\beta+\gamma]] \\
&\otimes_{\bold Q[[\beta+\gamma]]}
\eta^{\beta}(1-\eta)^{\gamma}
\bold Q[[\beta,\gamma]]).
\end{align*}
Therefore we have a homomorphism
$$
(f_1,f_2)^*:\bold B_{\Phi}(\alpha,\beta+\gamma)\otimes_{\bold Q[[\beta+\gamma]]}
\bold B_{\Phi}(\beta,\gamma) \to
H^2_{\Phi}(\Cal M_{5},\Cal F).
$$
\begin{lemma}
\label{transformation of Jacobi}
\begin{enumerate}
\item
\begin{align*}
&(f_1,f_2)^*
\big((1-x)^{\alpha}x^{\beta+\gamma}\eta^\beta(1-\eta)^{\gamma}
\frac{(\beta+\gamma)dx}{x}\wedge
\frac{\beta d\eta}{\eta}\big) \\
=&
(1-x)^{\alpha}y^{\beta}(x-y)^\gamma \alpha
\beta \frac{dx}{x-1}\wedge\frac{dy}{y} 
\end{align*}
\item
$(f_1,f_2)_*(D)=[0,1]\times [0,1]$
in $H^2_{dR}(\Cal M_{5},\Cal F_{dR})$.
\end{enumerate}
\end{lemma}
\begin{proof}
Since the element
$$
(1-x)^{\alpha}y^{\beta}(x-y)^\gamma 
(\alpha\frac{dx}{x-1}+\beta\frac{dy}{y}+\gamma\frac{d(x-y)}{x-y})
$$
is exact, we have the equality in
$H^2_{dR}(\Cal M_{5},\Cal F_{dR})$.
\end{proof}
\begin{proof}[Proof of Theorem \ref{one cocycle relation on phi beta}]
By Lemma \ref{transformation of Jacobi}, we have
\begin{align*}
&\int_D^{\Phi} (1-x)^{\alpha}y^{\beta}(x-y)^\gamma \alpha
\beta \frac{dx}{x-1}\wedge\frac{dy}{y}  \\
=&
\int^{\Phi}_{[0,1]}(1-x)^{\alpha}x^{\beta+\gamma}
\frac{(\beta+\gamma)dx}{x}
\int^{\Phi}_{[0,1]}
\eta^\beta(1-\eta)^{\gamma}
\frac{\beta d\eta}{\eta} \\
=&
B'_{\Phi}(\alpha,\gamma+\beta)\cdot
B'_{\Phi}(\gamma,\beta)
\end{align*}
Since the first integral is symmentric on $\alpha$ and $\beta$,
we have
$$
B'_{\Phi}(\alpha,\gamma+\beta)\cdot
B'_{\Phi}(\gamma,\beta)
=
B'_{\Phi}(\beta,\gamma+\alpha)\cdot
B'_{\Phi}(\gamma,\alpha).
$$
\end{proof}
\begin{definition}
\begin{enumerate}
\item
We define the Beta function $B(\alpha,\beta)$ by
$$
B_{\Phi}(\alpha,\beta)=
B'_{\Phi}(\alpha,\beta)
\frac
{\alpha+\beta}
{\alpha\beta}
$$
\item
We define $\Gamma_{\Phi}\in \bold C((x))$ by
$$
\Gamma_{\Phi}(x)=\frac{1}{x}\exp(\sum_{i=2}^{\infty}a_nx^n).
$$
Here $a_n\in \bold C$ is defined in Theorem 
\ref{one cocycle relation on phi beta}.
\item
The product $\Gamma_{\Phi}(a_1)\cdots \Gamma_{\Phi}(a_n)$
is denoted by $\Gamma_{\Phi}(a_1, \dots, a_n)$.
\end{enumerate}
\end{definition}
By the definition of $\Phi$-Gamma function and Proposition 
\ref{one cocycle relation on phi beta}, we have
$$
B_{\Phi}(\alpha,\beta)=
\frac
{\Gamma_{\Phi}(\alpha)\Gamma_{\Phi}(\beta)}
{\Gamma_{\Phi}(\alpha+\beta)}
$$

\subsection{Definition of $\Phi$-hypergeometric modules}
In this section, we define a framed $\Cal A_4^{\Phi}$-modules
at $\overline{01}$ associated to generalized hypergeometric functions.
\begin{lemma}
Let $x_1, \dots, x_k$ be the distinguished coordinates of $\Cal M_{k+3}$.
We set $\xi_1=x_1,\xi_2=\frac{x_2}{x_1},\dots,\xi_i=\frac{x_k}{x_{k-1}}$.
Then 
\begin{enumerate}
\item
$\xi_1,\dots, \xi_k$ and
$1-\xi_1,\dots, 1-\xi_k$ and $1-\xi_1\xi_2\cdots \xi_k$
are admissible functions of $\Cal M_{k+3}$.
\item
$
\frac{d\xi_1}{\xi_1},\dots,
\frac{d\xi_k}{\xi_k}$ are admissible one forms.
\end{enumerate}
\end{lemma}
We consider $\Cal M_{0,5}$ and $M_{0,6}$ with
the distinguished coordinates $(x_1,t)$ and $(x_1,x_2,t)$.
We define $\pi_5,\pi_6$ by
\begin{align*}
&\pi_5:\Cal M_{0,5}\to \Cal M_{0,4}:(x_1,t)\mapsto t \\
&\pi_6:\Cal M_{0,6}\to \Cal M_{0,4}:(x_1,x_2,t)\mapsto t.
\end{align*}

We define cell's $D_5$ and $D_6$ in 
$\Cal M_{0,5}(\bold R)_{\overline{01}}$ and
$\Cal M_{0,6}(\bold R)_{\overline{01}}$ by
\begin{align*}
D_5&=\{\xi_1\in [0,1], t=\overline{01}\},\\
D_6&=\{\xi_1, \xi_2 \in [0,1]\leq 1, t=\overline{01}\},
\end{align*}
We define $\Cal A_{5}^{\Phi}$ and $\Cal A_{6}^{\Phi}$ modules 
$\Cal F(a_1,a_2;b_1)$ and
$\Cal F(a_1,a_2,a_3;b_1,b_2)$ by
\begin{align*}
\xi_1^{a_1}
(1-\xi_1)^{b_1-a_1}
(1-t\xi_1)^{-a_2}
\bold Q[[a_1,a_2,b_1]].
\end{align*}
and
\begin{align*}
\xi_1^{a_1}
\xi_2^{a_2}
(1-\xi_1)^{b_1-a_1}
(1-\xi_2)^{b_2-a_2}
(1-t\xi_1\xi_2)^{-a_3}
\bold Q[[a_1,a_2,a_3,b_1,b_2]].
\end{align*}

\begin{definition}
We use the notation of cycles $\gamma_1,\gamma_2,\gamma_1^\#,\gamma_2^\#$
defined in (\ref{cycles on M_4}).

\begin{enumerate}
\item
We define hypergeomtric modules 
$HM(a_1,a_2,b_1)$ and
$HM(a_1,a_2,a_3,b_1,b_2)$ on $\Cal A_4^{\Phi}$ by
\begin{align*}
HM(a_1,a_2,b_1)=&\bold R^1\pi_{5*}\Cal F(a_1,a_2,b_1)\otimes \bold B_{\Phi}(a_1,b_1-a_1)^{-1} \\
HM(a_1,a_2,a_3,b_1,b_2)=&\bold R^2\pi_{6*}\Cal F(a_1,a_2,a_3,b_1,b_2) \\
&
\otimes \bold B{\Phi}(a_1,b_1-a_1)^{-1}\otimes \bold B_{\Phi}(a_2,b_2-a_2)^{-1}
\end{align*}
\item
We define hypergeometric function
$$
F_{\Phi}:\Cal A_{4,B,\overline{01}*}\to \bold C[[a_1,a_2,b_1]]
$$
by 
\begin{align*}
&F_{\Phi}(a_1,a_2,b_1+1,\gamma) \\
=&
B_{\Phi}(a_1,b_1-a_1+1)^{-1}
\int_{\gamma(\gamma_1)}^{\Phi}
\xi_1^{a_1}
(1-\xi_1)^{b_1-a_1}
(1-t\xi_1)^{-a_2}
\frac{d\xi_1}{\xi_1}
\end{align*}
for $\gamma \in \Cal A_{4,\overline{01}*}$.
\end{enumerate}
\end{definition}
\begin{proposition}
\label{Phi HGF and associator}
\begin{align*}
\left(
\begin{matrix} 1 & 0
\end{matrix}
\right)
\varphi(P_0,P_1)
\left(
\begin{matrix} 1 \\ 0
\end{matrix}
\right)
=&
B_{\Phi}(a,c-a+1)^{-1}
\int_{\varphi(D_5)}^{\Phi}
\xi_1^{a}
(1-\xi_1)^{c-a}
(1-t\xi_1)^{-b}
\frac{d\xi_1}{\xi_1} \\
\left(
\begin{matrix} 0 & 1
\end{matrix}
\right)
\varphi(P_0,P_1)
\left(
\begin{matrix} 1 \\ 0
\end{matrix}
\right)
=&
B_{\Phi}(a,c-a+1)^{-1} \\
&\frac{b}{a}\int_{\varphi(D_5)}^{\Phi}
t\xi_1^{a+1}
(1-\xi_1)^{c-a}
(1-t\xi_1)^{-b-1}
\frac{d\xi_1}{\xi_1}
\end{align*}
\end{proposition}
\begin{proof}
The regularized cycle $D_5$ defines an element in $HM^*_{B,\overline{01}}$.
By choosing the base
(\ref{hypergeometric base for de Rham}),
we have
\begin{align*}
\int_{[0,1]\times \overline{01}}^{\Phi}\omega_{2}=&0, \\
\int_{[0,1]\times \overline{01}}^{\Phi}\omega_1=&
\int_{[0,1]}^{\Phi}
\xi_1^{a}
(1-\xi_1)^{c-a}
(1-t\xi_1)^{-b}
\frac{d\xi_1}{\xi_1}\mid_{t=\overline{01}}
=B_{\Phi}(a,c-a+1).
\end{align*}
and as a conequence, we have
\begin{align*}
c_{HM^*}(\gamma_1(\overline{01}))=&
\left(
\begin{matrix}
1 \\ 0
\end{matrix}
\right)
B_{\Phi}(a,c-a+1).
\end{align*}
\end{proof}
\begin{proposition}
\label{dual initial data at 1}
Let $\gamma_1^\#,\gamma_2^\#$ be the cycles defined in
(\ref{cycles on M_4}).
Then
\begin{align*}
&c_{HM}(\gamma_1^\#(\overline{10}))=
\left(
\begin{matrix}
1 \\ \frac{-b}{a+b-c}
\end{matrix}
\right)B_{\Phi}(a,b-c),\\
&
c_{HM}(\gamma_2^\#(\overline{10}))=
\left(
\begin{matrix}
0 \\ \frac{a+b-c-1}{a}
\end{matrix}
\right)B_{\Phi}(c+1-a,1-b).
\end{align*}
\end{proposition}
\begin{proposition}
Let $\Phi$ be the associator
and $\Phi_0,\Phi_1$ be elements in $\bold C[[e_0,e_1]]$
defined by
$$
\Phi(e_0,e_1)=1+\Phi_0(e_0,e_1)e_0+\Phi_1(e_0,e_1)e_1
$$ 
Then 
$$B'_{\Phi}(-a,-b)=
\Phi_1^{ab}(a,b)b+1
$$ 
where $\Phi_0^{ab}(a,b)$ is the image
under the abelianization map
$\bold C\<\<e_0,e_1\>\>\to\bold C[[\alpha,\beta]]$.
As a consequence, we have $a_n=\zeta^{\Phi}(n)$.
\end{proposition}
\begin{proof}
By Proposition \ref{Phi HGF and associator},
Proposition \ref{dual initial data at 1},
and relation (\ref{cycles on M_4}),
we have
\begin{align*}
&
\left(
\begin{matrix} 0 & 1
\end{matrix}
\right) 
\Phi(P_0,P_1)
\left(
\begin{matrix} 1 \\ 0
\end{matrix}
\right) \\
=&B_{\Phi}(a,c-a+1)^{-1} \frac{b}{a}\int_{\gamma_1}^{\Phi}
t\xi_1^{a+1}
(1-\xi_1)^{c-a}
(1-t\xi_1)^{-b-1}
\frac{d\xi_1}{\xi_1}\mid_{t=\overline{01}}
\\
=&
B_{\Phi}(a,c-a+1)^{-1} 
\bigg(\frac{\bold s(b-c)}{\bold s(a+b-c)}\frac{-b}{a+b-c}B_{\Phi}(a,b-c)
\\
&+\frac{\bold s(b)}{\bold s(c-a-b)}\frac{a+b-c-1}{a}B_{\Phi}(c+1-a,1-b)
\bigg)
\end{align*}
We take a limit for $a\to 0$ and apply 
Proposition \ref{abelianization of end with index 1}.
Then we have
\begin{align*}
(-b)\Phi_1^{ab}(-c,c-b) 
=&
\lim_{a\to 0}
\left(
\begin{matrix} 1 & 0
\end{matrix}
\right)
\Phi(P_0,P_1)
\left(
\begin{matrix} 1 \\ 0
\end{matrix}
\right)
\\
=&\frac{b}{b-c}
(-1+\frac{\Gamma_{\Phi}(b-c+1,c+1)}
{\Gamma_{\Phi}(b+1)}),
\end{align*}
and
$$
\Phi^{ab}(-c,c-b)(c-b)+1=\frac{\Gamma_{\Phi}(b-c+1,c+1)}
{\Gamma_{\Phi}(b+1)}.
$$
\end{proof}
\subsection{Junction}
\subsubsection{Definition of Junctions}
Let $M, N$ be $\Cal A_4$ modules.
Let $pr_2:\Cal M_{4}\times \Cal M_{4} \to \Cal M_{4}:(x,y) \mapsto y$
be the second projection, $\Delta:\Cal M_{4} \to \Cal M_{4}\times \Cal M_{4}$
be the diagonal map, $i:\Cal M_{4} \to \Cal M_{4}\times \Cal M_{4}$
be an infinitesmal inclusion defined by $x\mapsto (x,\overline{01})$.
We consider the map
\begin{align*}
&\alpha:pr_2^*M \otimes pr_1^*(M^*\otimes N)
\to i^*
\big(pr_2^*M \otimes pr_1^*(M^*\otimes N)\big)\simeq M \otimes
(M^*\otimes N)_{\overline{01}} \\
&\beta:pr_2^*M \otimes pr_1^*(M^*\otimes N)
\to \Delta^*
\big(pr_2^*M \otimes pr_1^*(M^*\otimes N)\big)\simeq M \otimes
M^*\otimes N\xrightarrow{ev} N 
\end{align*}
Here $ev:M \otimes M^*\otimes N \to N$
is given by the evaluation map.
Then we have a complex $\bold E(M,N)$:
$$
\bold E(M,N):pr_2^*M \otimes pr_1^*(M^*\otimes N)
\xrightarrow{\alpha\oplus \beta} i_*(M \otimes
(M^*\otimes N)_{\overline{01}}) \oplus \Delta_*N
$$
We define the junction $E(M,N)=\bold R^1pr_{2*}\bold E(M,N)$

\subsubsection{Gauss-Manin connection for a junction}
We consider the Gauss-Manin connection on $E(M,N)_{dR}$.
Let 
\begin{align*}
& P=P_0\frac{dx}{x}+P_1\frac{dx}{x-1}:
M_{dR}\to M_{dR}\otimes \<\frac{dx}{x},\frac{dx}{x-1}\> \\
& Q=Q_0\frac{dx}{x}+Q_1\frac{dx}{x-1}:
N_{dR}\to N_{dR}\otimes \<\frac{dx}{x},\frac{dx}{x-1}\> 
\end{align*}
be the associated connection of $M$ and $N$. 
By choosing basis $\{\omega_i\}$ and $\{\eta_j\}$ of $M_{dR}$ and $N_{dR}$,
the map $P_0,\dots, Q_1$ can be expressed as matrices by the rule:
\begin{align}
\label{matrix rule for de Rham cohomologies}
\nabla 
\left(
\begin{matrix} \omega_1 \\ \vdots \\ \omega_n
\end{matrix}
\right)
=P
\left(
\begin{matrix} \omega_1 \\ \vdots \\ \omega_n
\end{matrix}
\right)
\end{align}
The module $\bold E(M,N)_{dR}$ is quasi-isomorphic to the 
associate simple complex of the following complex $\bold E_{dR}$
$$
\begin{matrix}
M_{dR} \otimes M^*_{dR}\otimes N_{dR} &\xrightarrow{-1\otimes^tP\otimes 1+1\otimes1\otimes Q }
&M_{dR} \otimes M^*_{dR}\otimes N_{dR} \otimes \<\frac{dx}{x},\frac{dx}{x-1}\>
\\
\downarrow{id\oplus ev}  & &
\\
(M_{dR} \otimes
M_{dR}^*\otimes N_{dR})  & & \\
 \oplus N_{dR} & & 
\end{matrix}
$$
Therefore
\begin{align}
\label{expression of convolution}
& H^1(\bold E_{dR}) \\
\nonumber
\simeq &
M_{dR} \otimes M^*_{dR}\otimes N_{dR} \frac{dx}{x}
\oplus
M_{dR} \otimes M^*_{dR}\otimes N_{dR} \frac{dx}{x-1}
\oplus N_{dR}
\end{align}
Under this isomorphism, the Gauss-Manin connection 
$$
\nabla:
H^1(\bold E_{dR})\to H^1(\bold E_{dR})\otimes \<\frac{dy}{y},\frac{dy}{y-1}\>
$$
is given by
\begin{align}
\label{formula for convolution connection}
&\nabla(u_0\frac{dx}{x}+u_1\frac{dx}{x-1}+
v)
\\
\nonumber
= &(
(P\otimes 1 \otimes 1)(u_0)\frac{dx}{x}+
(P\otimes 1 \otimes 1)(u_1)\frac{dx}{x-1}+
ev(u_0)\frac{dy}{y}+ev(u_1)\frac{dy}{y-1}+Q(v))
\\
\nonumber
= &\bigg(
(P_0\otimes 1 \otimes 1)(u_0)\frac{dx}{x}+
(P_0\otimes 1 \otimes 1)(u_1)\frac{dx}{x-1}+
ev(u_0)+Q_0(v)\bigg)\frac{dy}{y}
\\
\nonumber
&+\bigg(
(P_1\otimes 1 \otimes 1)(u_0)\frac{dx}{x}+
(P_1\otimes 1 \otimes 1)(u_1)\frac{dx}{x-1}+
ev(u_1)+Q_0(v)\bigg)\frac{dy}{y-1}
\end{align}
for $u_0, u_1 \in M_{dR} \otimes M^*_{dR}\otimes N_{dR}$
and
$v \in N_{dR}$.

Bases of 
$M_{dR}\otimes M_{dR}^*\otimes N_{dR}\frac{dx}{x}$,
$M_{dR}\otimes M_{dR}^*\otimes N_{dR}\frac{dx}{x-1}$,
and $N_{dR}$ form a basis of
$E(M,N)_{dR}$.
Using this basis,
the connection $\nabla$ on $E(M,N)_{dR}$
can be expressed as 
$R_0\frac{dy}{y}+R_1\frac{dy}{y-1}$
via the rule 
(\ref{matrix rule for de Rham cohomologies}),
where 
\begin{align*}
&R_0=\left(
\begin{matrix}
P_0\otimes 1 \otimes 1 & 0 & Ev \\
0& P_0\otimes 1 \otimes 1 &  0 \\
0 & 0 & Q_0
\end{matrix}
\right), \\
&R_1=\left(
\begin{matrix}
P_1\otimes 1 \otimes 1 & 0 & 0 \\
0& P_1\otimes 1 \otimes 1 &  Ev \\
0 & 0 & Q_1
\end{matrix}
\right).
\end{align*}
by the formula (\ref{formula for convolution connection}).

\subsubsection{Horizontal section of the dual}

The action of $\varphi(e_0,e_1) \in \Cal A_{4,dR}=
\bold C\<\<e_0,e_1\>\>$ on $H^1(E_{dR})^*$ is given
by the left multiplication of the matrix $\varphi(R_0,R_1)$.
For $I=(i_1, \dots, i_n)\in \{0,1\}^n$,
we set
$$
P_I=P_{i_1}\cdots P_{i_n},\quad
R_I=R_{i_1}\cdots R_{i_n},etc
$$
By (\ref{matrix rule for de Rham cohomologies}),
we have the following proposition.
\begin{proposition}
We have
\begin{align*}
R_I(v^*)
=
&\sum_{I=J_10J_2}((P_{J_1}\otimes 1\otimes 1)Ev(Q_{J_2}(v^*))
\frac{dx}{x}^*+
\\
&+\sum_{I=J_11J_2}((P_{J_1}\otimes 1\otimes 1)Ev(Q_{J_2}(v^*))
\frac{dx}{x-1}^*+Q_I(v^*)
\end{align*}
for $u_0^*,u_1^* \in M^*_{dR} \otimes M_{dR}\otimes N_{dR}^*$
and
$v^* \in N_{dR}^*$.

As a consequnece,
for $\varphi$ given in (\ref{typical element}), we have
\begin{align*}
\varphi(R_0,R_1)(v^*)
=&
\sum_{J_1,J_2}c_{J_10J_2}
((P_{J_1}\otimes 1\otimes 1)Ev(Q_{J_2}(v^*))
\frac{dx}{x}^*+
\\
&
+\sum_{J_1,J_2}c_{J_11J_2}
((P_{J_1}\otimes 1\otimes 1)Ev(Q_{J_2}(v^*))
\frac{dx}{x-1}^*+
\varphi(Q_0,Q_1)(v^*)
\end{align*}
\end{proposition}

\subsubsection{Betti part of the dual}
Using the chain complex 
the dual local sysstem of the Betti-part 
$(E(M,N)_B)^*$ of the junction $E(M,N)$ 
is naturally isomorphic to the cohomology of the 
associate simple complex of the following chain complex.
$$
\begin{matrix}
C_{\bullet}(\Cal M_4/\Cal M_4, M^*_{y} \otimes
M_{\overline{01}}\otimes N_{\overline{01}}^*)  \\
 \oplus C_{\bullet}(\Cal M_4/\Cal M_4,N_{x}^*)
\end{matrix}
\xrightarrow {i_*\oplus ev_*}  
C_{\bullet}(\Cal M_5/\Cal M_4, M_{y}^* \otimes M_{x}\otimes N_{x}^*) 
$$

\begin{definition}
\label{definition of junction cycle}
Let $y\in [0,1]$,$\tau \in N_y^*$ and $\Delta\in M_y^*\otimes M_y$ be the element
corresponding to the identity element.
The local section of $pr_1^*M^*\otimes pr_2^*(M\otimes N)$
on $\{(x,y)\mid 0\leq x \leq y\}$ whose fiber at $(y,y)$ is equal to
$\Delta\otimes \tau \in M_y^*\otimes M_y\otimes N^*_y$
to  is denoted by $\delta(\tau)$.
The element
\begin{align*}
& \delta(\tau)_{\{x\mid 0\leq x \leq y\}\times \{y\}}
+\tau_{(y,y)} - 
\delta(\tau)_{(\overline{01},y)}
\\
&\in
C_{1}(pr_2^{-1}(y), M_{y}^* \otimes M\otimes N^*) 
\oplus C_{0}(y\times y, N_{y}^*) 
\oplus C_{0}(\overline{01}\times y, M_{y}^* \otimes M_{\overline{01}}\otimes 
N_{\overline{01}}^*) 
\end{align*}
is closed and defines an element $J(y,\tau)$ of $E(M,N)_{B,y}^*$, which is called
the junction cycle for $\tau$.
\end{definition}
If $\tau_{\overline{01}}$ and $\tau_{\overline{10}}$ are fibers
of a local section $\tau$ of $N^*$ on $[0,1]$, then
we have $[0,1]J(\overline{01},\tau_{\overline{01}})=
J(\overline{10},\tau_{\overline{10}})$ using the action of $\Cal A_{4,B}$.
In this situation, $J(y,\tau_y)$ is denoted by $J(y,\tau)$.
\begin{proposition}
\label{proposition of convolution formula}
Let $\tau$ be a local section of $N^*$ on $[0,1]$.
\begin{enumerate}
\item
$c_E(J(\overline{01}),\tau)\in N_{dR}^*$
\item
We set $c_E(J(\overline{01},\tau))=v^*$.
Let $\displaystyle u_0\frac{dx}{x}\in E(M,N)_{dR}$.
Then we have
\begin{align}
\label{convolition formula for matrix elements}
&c_E(J(\overline{10},\tau))(u_0\frac{dx}{x})
=\sum_{J_1,J_2}c_{\Phi,J_10J_2}
((P_{J_1}\otimes 1\otimes 1)Ev(Q_{J_2}(v^*))(u_0), 
\end{align}
\end{enumerate}
\end{proposition}
\begin{proof}
(1)
We have
\begin{align*}
J(\overline{01},\tau)=&
\tau_{(\overline{01}, \overline{01})} 
- \delta(\tau)_{(\overline{01}, \overline{01})}
\\
&\in
C_{0}(\overline{01}\times \overline{01}, N_{\overline{01}}^*)
\oplus 
C_{0}(\overline{01}\times \overline{01}, M_{\overline{01}}^* 
\otimes M_{\overline{01}}\otimes 
N_{\overline{01}}^*).
\end{align*}
Therefore for $u_0,u_1 \in M_{dR} \otimes M^*_{dR}\otimes N_{dR}$
and
$v \in N_{dR}$,
$$
c_E(J(\overline{01},\tau))(u_0\frac{dx}{x}+u_1\frac{dx}{x-1}+
v)=c_N(\tau)(v).
$$
Therefore $c_E(J(\overline{01},\tau))\in N^*_{dR}$

(2)
\begin{align*}
c_E(J(\overline{10},\tau))(u_0\frac{dx}{x})=&
c_E([01])c_E(J(\overline{01}))(u_0\frac{dx}{x})
\\
=&
c_E([01])(v^*)(u_0\frac{dx}{x})
\\
=&\sum_{J_1,J_2}c_{\Phi,J_10J_2}
((P_{J_1}\otimes 1\otimes 1)Ev(Q_{J_2}(v^*))(u_0)
\end{align*}
\end{proof}

\section{Generating function and Zagier's expression}
\label{section:Comparison to Zagier}
In this section, we show that the formal power series $\Phi(a,b)$
defined in the last section coincides with the formal power series
defined in Zagier's paper.
\subsection{Junction for hypergeomtric modules}
Let $V=\bold Q((a,b))b_1\oplus \bold Q((a,b))b_1$ 
be a free $\bold Q((a,b))$ module generated by two basis
$$
b_1=
\left(
\begin{matrix}
1 & 0
\end{matrix}
\right),
b_2=
\left(
\begin{matrix}
0 & 1
\end{matrix}
\right).
$$
The $\bold Q((a,b))$-dual of $V$ is denoted by $V^*$
and the dual basis are denoted by 
$$
b_1^*=
\left(
\begin{matrix}
1 \\ 0
\end{matrix}
\right),
b_2^*=
\left(
\begin{matrix}
0 \\ 1
\end{matrix}
\right).
$$
We apply Proposition 
\ref{proposition of convolution formula}
 by setting
\begin{align}
\label{P twisted and Q}
& M=HM(a_1,b_1,c_1)\otimes x^{-u},\quad N=HM(a_2,b_2,c_2),\\
\nonumber
&
P_0=
\left(\begin{matrix}
-u & a_1 \\
0 & -u-c_1
\end{matrix}
\right),
P_1=\left(
\begin{matrix}
0 & 0 \\
-b_1 & c_1-a_1-b_1
\end{matrix}\right), \\
& \nonumber
Q_0=\left(
\begin{matrix}
0 & a_2 \\
0 & -c_2
\end{matrix}\right),
Q_1=
\left(\begin{matrix}
0 & 0 \\
-b_2 & c_2-a_2-b_2
\end{matrix}
\right), 
\end{align}
Then the element $\Delta$ in Definition \ref{definition of junction cycle}
is equal to
\begin{align*}
\Delta=&\frac{a_1\bold s(-a_1)\bold s(b_1-c_1)}{\bold s(a_1+b_1-c_1)}
\gamma_1y^{-u} \otimes \gamma_1^*x^u+
\frac{a_1\bold s(a_1-c_1)\bold s(b_1)}{\bold s(c_1-a_1-b_1)}
\gamma_2y^{-u}\otimes \gamma_2^*x^u \\
&\in 
M_{B}^*\otimes M_{B}.
\end{align*}
by the relation (\ref{topological cycle givein exponential}).
Here $\gamma_i,\gamma_i^*$ are the topological cycles
corresponding to the base
\S \ref{sec:Gauss-Manin connection and horizontal section on the daul}
and \S \ref{subsec: Dual differential equation GM connection}.
We choose a local section $\gamma$ of $N^*_B$ on $[0,1]$ such
that the fiber of $\gamma(\overline{01})$ at $\overline{01}$
goes to
$c_{N^*}(\gamma(\overline{01}))=b_1^* \in N^*_{dR}$ via the comparison map $c_N$.
We apply Proposition 
Proposition \ref{proposition of convolution formula} 
by setting $u_0=b_1\otimes b_1^*\otimes b_1$

\subsubsection{Using Hochschild-Serre-Fubini theorem}
Using Hochschild-Serre-Funibi theorem, we have the following
lemma.
\begin{proposition}
\label{equality for generateing function and hypergeometric function}
\begin{align*}
c_E(J(\overline{10},\gamma))((b_1\otimes b^*_1\otimes b_1) \frac{dx}{x})
=&
 \int_{(0,1)}
F_{\Phi}(-a_1,-b_1;c_1-a_1-b_1+1;1-x)x^u \\
&
F_{\Phi}(a_2,b_2;c_2+1;x) 
\frac{dx}{x}.
\end{align*}
\end{proposition}
\begin{proof}
We compute the pairing 
$I_i=\<\widetilde{\gamma_i},b_1\otimes b_1^*\otimes b_1\frac{dx}{x}\>$ 
for $i=1,2$ at the fiber at $y=\overline{01}$.
\begin{align*}
&\widetilde{\gamma}=\{(t_1,t_2,t_3,x)\mid 
t_1,t_2,t_3 \in [0,1],x\in (0,1)\} 
\end{align*}
Then we have
\begin{align*}
I_1(y)=
&\frac{a_1\bold s(-a_1)\bold s(b_1-c_1)\Gamma(c_2+1)}
{\bold s(a_1+b_1-c_1)\Gamma(a_2,c_2-a_2+1)} \\
&\int_{\widetilde{\gamma}}
t_0^{a_1-1}(1-t_0)^{b_1-c_1-1}(1-(1-y)t_0)^{-b_1}y^{-u} \\
&t_1^{-a_1-1}(1-t_1)^{c_1-b_1}(1-(1-x)t_1)^{b_1}x^u \\
&t_2^{a_2-1}(1-t_2)^{c_2-a_2}(1-xt_2)^{-b_2} \frac{1}{x}dt_0dt_1dt_2dx
\end{align*}
\begin{align*}
I_2(y)=(\text{const.})\int_{\widetilde{\gamma}}
&(1-y)^{c_1-b_1-a_1+1}
t_0^{c_1-a_1-1}(1-t_0)^{-b_1}(1-(1-y)t_0)^{b_1-c_1-1}y^{-u} \\
&(1-x)^{-c_1+a_1+b_1}t_1^{a_1-c_1-1}(1-t_1)^{b_1}(1-(1-x)t_1)^{-b_1+c_1}x^u \\
&t_2^{a_2-1}(1-t_2)^{c_2-a_2}(1-xt_2)^{-b_2} \frac{1}{x}dt_0dt_1dt_2dx
\end{align*}
Therefore by integrating $t_0$ first, we have $I_2(\overline{10})=0$ and 
\begin{align*}
I_1(\overline{01})=&\frac{\Gamma(c_1-a_1-b_1+1,c_2+1)}
{\Gamma(-a_1,c_1-b_1+1,a_2,c_2-a_2+1)}  \\
&\int_{D}
t_1^{-a_1-1}(1-t_1)^{c_1-b_1}(1-(1-x)t_1)^{b_1}x^u \\
&t_2^{a_2-1}(1-t_2)^{c_2-a_2}(1-xt_2)^{-b_2} \frac{1}{x}dt_1dt_2dx,
\end{align*}
where
$
D=\{(t_1,t_2,x)\mid 
t_1,t_2 \in [0,1],x\in (0,1)\} 
$.

\end{proof}

\subsection{A classical integral formula}
To compute the integral of Prposition \ref{equality for 
generateing function and hypergeometric function},
we need an integral formula for assiciators
in Theorem \ref{main identity for 4F3}.
Before proving the integral formula 
(\ref{key integral formula})
for associators,
we recall a proof of the corresponding classical integral formula.
\begin{lemma}
\label{lemma Gauss HGF and integral}
\begin{enumerate}
\item
\begin{align*}
&\int_{[0,1]^2}s_1^{b_2-a_1-1}(1-s_1)^{a_1-1}
\ _2F_1(a_2,a_3;a_1;1-s_1)
ds_1 
\\
=&
\frac{\Gamma(a_1)\Gamma(b_2-a_1)\Gamma(b_2-a_2-a_3)}
{\Gamma(b_2-a_2)\Gamma(b_2-a_3)}.
\end{align*}
\item
\begin{align*}
&\int_{[0,1]}\ _2F_1(p_1,p_2;q_1,us) 
s^{b_2-a_1-1}(1-s)^{a_1-1}
\ _2F_1(a_2,a_3;a_1;1-s)ds
\\
=&
\frac{\Gamma(a_1)\Gamma(b_2-a_1)\Gamma(b_2-a_2-a_3)}
{\Gamma(b_2-a_2)\Gamma(b_2-a_3)} \\
& \ _4F_3(p_1,p_2,b_2-a_1,b_2-a_2-a_3;
q_1,b_2-a_2,b_2-a_3;u)
\end{align*}
\end{enumerate}
\end{lemma}
\begin{proof}
(1)
\begin{align*}
& 
\frac{\Gamma(b_2)}{\Gamma(a_1)\Gamma(b_2-a_1)}
\int_{[0,1]^2}s_1^{a_1-1}(1-s_1)^{b_2-a_1-1}
\ _2F_1(a_2,a_3;a_1;s_1)
ds_1 \\
=& 
\frac{\Gamma(b_2)}{\Gamma(a_1)\Gamma(b_2-a_1)}
\frac{\Gamma(a_1)}{\Gamma(a_2)\Gamma(a_1-a_2)}
\\
&\int_{[0,1]^2}s_1^{a_1-1}(1-s_1)^{b_2-a_1-1}s_2^{a_2-1}(1-s_2)^{a_1-a_2-1}
(1-s_1s_2)^{-a_3}ds_1ds_2 \\
=&\ _3F_2(a_1,a_2,a_3;b_2,a_1;1) =\ _3F_2(a_1,a_2,a_3;a_1,b_2;1) 
\\
=&\ _2F_1(a_2,a_3;b_2;1)=\frac{\Gamma(b_2)\Gamma(b_2-a_2-a_3)}{\Gamma(b_2-a_2)\Gamma(b_2-a_3)} 
\end{align*}

(2)
By the integral expression of hypergeometric function, we have
\begin{align*}
&\int_{[0,1]}\ _2F_1(p_1,p_2;q_1,us) 
s^{b_2-a_1-1}(1-s)^{a_1-1}
\ _2F_1(a_2,a_3;a_1;1-s)ds
\\
=&\sum_m\frac{(p_1)_m(p_2)_m}{(q_1)_m}u^m
s^{m+b_2-a_1-1}(1-s)^{a_1-1}
\ _2F_1(a_2,a_3;a_1;1-s)ds.
\end{align*}
By (1),
the above is equal to 
\begin{align*}
&\sum_m\frac{(p_1)_m(p_2)_m}{(q_1)_m}u^m
\frac{\Gamma(a_1)\Gamma(m+b_2-a_1)\Gamma(m+b_2-a_2-a_3)}
{\Gamma(m+b_2-a_2)\Gamma(m+b_2-a_3)}
\\
=&
\frac{\Gamma(a_1)\Gamma(b_2-a_1)\Gamma(b_2-a_2-a_3)}
{\Gamma(b_2-a_2)\Gamma(b_2-a_3)}
\sum_m\frac{(p_1)_m(p_2)_m(b_2-a_1)_m(b_2-a_2-a_3)_m}
{(q_1)_m(b_2-a_2)_m(b_2-a_3)_m}u^m
\\
=&
\frac{\Gamma(a_1)\Gamma(b_2-a_1)\Gamma(b_2-a_2-a_3)}
{\Gamma(b_2-a_2)\Gamma(b_2-a_3)} \\
& \ _4F_3(p_1,p_2,b_2-a_1,b_2-a_2-a_3;
q_1,b_2-a_2,b_2-a_3;u).
\end{align*}
\end{proof}
\subsection{Zagier's generating function for associators}

\subsubsection{An integral formula for associators}
Let $u_1,u_2,u_3$ be the distinguished coordinate of $\Cal M_6$. 
We define admissible functions $s_1,s_2,s_3,x_1,x_2,x_3$ by
\begin{align}
\label{admissible functions of M7}
&x_1=\frac{(u_1-1)(u_2-u_3)}{(u_1-u_3)(u_2-1)},\quad
x_2=\frac{(u_2-0)(u_1-u_3)}{(u_2-u_3)(u_1-0)},\quad
x_3=\frac{(u_1-u_3)(\infty-u_2)}{(u_1-u_2)(\infty-u_3)}, \\
\nonumber
&s_1=\frac{(u_1-1)(\infty-0)}{(u_1-0)(\infty-1)},
\quad s_2=\frac{(u_2-0)(\infty-1)}{(u_2-1)(\infty-0)},\quad 
s_3=\frac{(0-u_3)(\infty-1)}
{(0-1)(\infty-u_3)}.
\end{align}
Then we have
\begin{align*}
&
1-s_1=\frac{1}{u_1}=\frac{x_3^*(1-x_2^*x_3^*)}{1-x_1^*x_3},
\quad 1-x_1x_2=1-s_1s_2=
\frac{(u_2-u_1)(0-1)}{(u_2-1)(0-u_1)}
\end{align*}
and
\begin{align*}
\frac{x_1}{x_1^*x_3^*}=\frac{s_1}{s_1^*s_3^*},\quad
\frac{x_2}{x_2^*x_3}=\frac{s_2}{s_2^*s_3},\quad
x_1^*x_3&=s_2^*s_3^*,\quad
x_2^*x_3^*=s_1^*s_3.
\end{align*}
As a consequence, we have
\begin{align}
\label{correspondence inducing symmetry}
&w_1^{p_1}(1-w_1)^{q_1-p_1}(1-\frac{w_1}{u_1})^{-p_2} s_1^{a_1}(1-s_1)^{b_2-a_1}
 s_2^{a_2}(1-s_2)^{b_1-a_2} \\
\nonumber
&s_3^{b_2-a_2}(1-s_3)^{b_1-a_1}
(1-s_1s_2)^{-a_3}
\\
\nonumber
=&w_1^{p_1}(1-w_1)^{q_1-p_1}(
1-\frac{w_1}{u_1})^{-p_2} (1-x_1)^{b_1-a_1}x_1^{a_1}
(1-x_2)^{b_2-a_2}x_2^{a_2} 
\\
\nonumber
&(1-x_3)^{b_2-a_1}x_3^{b_1-a_2} 
(1-x_1x_2)^{-a_3}
\end{align}
and
$$
\frac{dx_1}{x_1}
\frac{dx_2}{x_2}
dx_3=
\frac{ds_1}{s_1}
\frac{ds_2}{s_2}
ds_3
$$

The main theorem in this section is the following:
\begin{theorem}
\label{main identity for 4F3}
\begin{enumerate}
\item
\begin{align}
\label{key integral formula}
&\int_{[0,1]}^{\Phi} F_{\Phi}(p_1,p_2;q_1,s) 
s^{b_2-a_1-1}(1-s)^{a_1-1}
F_{\Phi}(a_2,a_3;a_1;1-s)ds
\\
\nonumber
=&
\frac{\Gamma_{\Phi}(a_1,b_2-a_1,b_2-a_2-a_3)}
{\Gamma_{\Phi}(b_2-a_2,b_2-a_3)} \\
\nonumber
& F_{\Phi}(p_1,p_2,b_2-a_1,b_2-a_2-a_3;
q_1,b_2-a_2,b_2-a_3;1)
\end{align}
\item
\begin{align*}
&\int_{(0,1)}^{\Phi} 
F_{\Phi}(a_2,a_3;1;1-s)
s^{b_2-1}
F_{\Phi}(p_1,p_2;q_1+1,s)
ds
\\
=&\frac{\Gamma_{\Phi}(b_2,b_2+1-a_2-a_3)}
{\Gamma_{\Phi}(b_2+1-a_2,b_2+1-a_3)} 
\\
& \times F_{\Phi}(p_1,p_2,b_2,b_2+1-a_2-a_3;
q_1+1,b_2+1-a_2,b_2+1-a_3;1) 
\\
&-\frac{
\Gamma_{\Phi}(1-a_2-a_3)}
{b_1\Gamma_{\Phi}(1-a_2,1-a_3)} 
\end{align*}
\end{enumerate}
\end{theorem}
\begin{proof}
Let $w_1,u_1,u_2,u_3$ be the distinguished coordinate of $\Cal M_7$.
We consider admissible functions 
$s_1,s_2,s_3,x_1,x_2,x_3$ 
on $u_1,u_2,u_3$ define in 
(\ref{admissible functions of M7}).
Using the relation, 
(\ref{correspondence inducing symmetry}),
we have the following equatlity:
\begin{align}
&
\label{S6 first integral}
\int_{[0,1]}^{\Phi}
w_1^{p_1}(1-w_1)^{q_1-p_1-1}
(1-\frac{w_1}{u_1})^{-p_2}s_1^{a_1}(1-s_1)^{b_2-a_1-1}
\\ 
\nonumber
&
 s_2^{a_2}(1-s_2)^{b_1-a_2-1} 
s_3^{b_2-a_2-1}(1-s_3)^{b_1-a_1-1}
(1-s_1s_2)^{-a_3}ds_3\frac{ds_1ds_2dw_1}{s_1s_2w_1} \\
\label{S6 second integral}
=&\int_{[0,1]}^{\Phi}
w_1^{p_1}(1-w_1)^{q_1-p_1-1}(1-\frac{w_1}{u_1})^{-p_2}(1-x_1)^{b_1-a_1-1}x_1^{a_1}
\\ 
\nonumber
&
(1-x_2)^{b_2-a_2-1}x_2^{a_2}
(1-x_3)^{b_2-a_1-1} x_3^{b_1-a_2-1} 
(1-x_1x_2)^{-a_3}\frac{dx_1dx_2dx_3dw_1}{x_1x_2w_1}
\end{align}
We multiply $(b_1-a_1+1)$ with
(\ref{S6 first integral}) and
(\ref{S6 second integral})
and take a limit where $b_1$ tends to $a_1-1$.
Using
$
\lim_{b_1\to a_1}
\displaystyle (b_1-a_1)B_{\Phi}(b_2-a_2,b_1-a_1)=1,
$
and $\frac{1}{u_1}=1-s_1$,
the limit of (\ref{S6 first integral}) is equal to
\begin{align}
\label{first limit}
&\lim_{b_1\to a_1}(b_1-a_1)\int_{[0,1]}^{\Phi}
w_1^{p_1}(1-w_1)^{q_1-p_1-1}(1-\frac{w_1}{u_1})^{-p_2}
\\ 
\nonumber
&s_1^{a_1}(1-s_1)^{b_2-a_1-1}
 s_2^{a_2}(1-s_2)^{b_1-a_2-1} s_3^{b_2-a_2-1}(1-s_3)^{b_1-a_1-1}
(1-s_1s_2)^{-a_3}ds_3\frac{ds_1ds_2dw_1}{s_1s_2w_1} 
\\
\nonumber
=&\int_{[0,1]}^{\Phi}
w_1^{p_1}(1-w_1)^{q_1-p_1-1}(1-w_1(1-s_1))^{-p_2}
\\ 
\nonumber
&s_1^{a_1}(1-s_1)^{b_2-a_1-1}
 s_2^{a_2}(1-s_2)^{a_1-a_2-1} 
(1-s_1s_2)^{-a_3}\frac{ds_1ds_2dw_1}{s_1s_2w_1} 
\\
\nonumber
=&
\frac{\Gamma_{\Phi}(p_1,q_1-p_1,a_2,a_1-a_2)}
{\Gamma_{\Phi}(q_1,a_1)}
\\
\nonumber
&\int_{[0,1]}^{\Phi}F_{\Phi}(p_1,p_2;q_1,1-s_1) 
s_1^{a_1-1}(1-s_1)^{b_2-a_1-1}
F_{\Phi}(a_2,a_3;a_1;s_1)ds_1 \\
\nonumber
=&
\frac{\Gamma_{\Phi}(p_1,q_1-p_1,a_2,a_1-a_2)}
{\Gamma_{\Phi}(q_1,a_1)}
\\
\nonumber
&\int_{[0,1]}^{\Phi}F_{\Phi}(p_1,p_2;q_1,s) 
s^{b_2-a_1-1}(1-s)^{a_1-1}
F_{\Phi}(a_2,a_3;a_1;1-s)ds
\end{align}
We compute the limit of (\ref{S6 second integral})
using Proposition \ref{limit tends to delta function}.
Since $\lim_{x_1\to 1}\frac{1}{u_1}=x_2^*x_3^*$, 
we have
\begin{align}
\label{second limit}
&\lim_{b_1\to a_1}(b_1-a_1)\int_{[0,1]}^{\Phi}
w_1^{p_1}(1-w_1)^{q_1-p_1-1}(1-\frac{w_1}{u_1})^{-p_2}
(1-x_1)^{b_1-a_1-1}x_1^{a_1}
\\ 
\nonumber
&
(1-x_2)^{b_2-a_2-1}x_2^{a_2}
(1-x_3)^{b_2-a_1-1}x_3^{b_1-a_2-1} 
(1-x_1x_2)^{-a_3}\frac{dx_1dx_2dx_3dw_1}{x_1x_2w_1}
\\
\nonumber
=&\int_{[0,1]}^{\Phi}
w_1^{p_1}(1-w_1)^{q_1-p_1-1}(1-w_1x_2^*x_3^*)^{-p_2}
\\ 
\nonumber
&
(1-x_2)^{b_2-a_2-a_3-1}x_2^{a_2}
(1-x_3)^{b_2-a_1-1}x_3^{a_1-a_2} 
\frac{dx_2dx_3dw_1}{x_2x_3w_1} 
\\
\nonumber
=&\int_{[0,1]}^{\Phi}
w_1^{p_1}(1-w_1)^{q_1-p_1-1}(1-w_1x_2x_3)^{-p_2}
\\ 
\nonumber
&
x_2^{b_2-a_2-a_3}(1-x_2)^{a_2-1}
x_3^{b_2-a_1}(1-x_3)^{a_1-a_2-1} 
\frac{dx_2dx_3dw_1}{x_2x_3w_1} 
\\
\nonumber
=&
\frac{\Gamma_{\Phi}(p_1,q_1-p_1,b_2-a_2-a_3,a_2,b_2-a_1,a_1-a_2)}
{\Gamma_{\Phi}(q_1,b_2-a_3,b_2-a_2)} \\
\nonumber
& F_{\Phi}(p_1,b_2-a_2-a_3,b_2-a_1,p_2;
q_1,b_2-a_3,b_2-a_2;1)
\end{align}
By comparing two limits
(\ref{first limit}) and (\ref{second limit}),
we have the theorem.

(2) By replacing $q_1$ and $a_1$ by $q_1+1$ and $1$,
we have
\begin{align*}
&\int_{[0,1]}^{\Phi} 
F_{\Phi}(a_2,a_3;1;1-s)
s^{b_2-1}
F_{\Phi}(p_1,p_2;q_1+1,s)
ds
\\
=&\frac{\Gamma_{\Phi}(b_2,b_2+1-a_2-a_3)}
{\Gamma_{\Phi}(b_2+1-a_2,b_2+1-a_3)} 
\\
& F_{\Phi}(p_1,p_2,b_2,b_2+1-a_2-a_3;
q_1+1,b_2+1-a_2,b_2+1-a_3;1)
\end{align*}
By Proposition \ref{difference between limit and regularization}, 
we have the statement (2).
\end{proof}
As a corollary, we have the following corollary.
\begin{corollary}
Let $P, Q$ be matrices in
(\ref{P twisted and Q}) evaluated at $c_1=0$.
Then we have

\begin{align*}
&\sum_{J_1,J_2,\deg(J_2)>0}c_{\Phi,J_10J_2}
(b_1,P_{J_1}b_1^*b_1Q_{J_2}b_1^*))
\\
=&\frac{\Gamma_{\Phi}(u,u+1-a_1-b_1)}
{\Gamma_{\Phi}(u+1-a_1,u+1-b_1)} 
\\
& \bigg[F_{\Phi}(a_2,b_2,u,u+1-a_1-b_1;
c_2+1,u+1-a_1,u+1-b_1;1)-1 \bigg]
\end{align*}
\end{corollary}
By setting $a_2=a, b_2=-a, a_1=-b, b_1=b$ and
taking the limit for $u\to 0$,
we have the following theorem.
\begin{theorem}
\label{first identity by Zagier}
We have the following equality:
\begin{align*}
\sum_{n\geq 0,m>0}c_{\Phi,(01)^n0(01)^m}b^{2n+1}a^{2m}=\bold s(b)
\frac{d}{dz}\mid_{z=0}F_{\Phi}(a,-a,z;1+b,1-b;1).
\end{align*}
\end{theorem}

\section{Selberg integral and Dixon's theorem}

\subsection{$\Phi$-Selberg integral formula}
Let $x_1,x_2$ be the distinguished coordinate of $\Cal M_{5}$.
We define even and odd $\Phi$-Selberg integrals by
\begin{align*}
S_{\Phi}^+(a,b,c)
=
\int^{\Phi}_{0\leq x_2 \leq x_1\leq 1}(x_1x_2)^{a-1}((1-x_1)(1-x_2))^{b-1}
(x_1-x_2)^{2c}
dx_1dx_2 \\
S_{\Phi}^-(a,b,c)
=
\int^{\Phi}_{0\leq x_2 \leq x_1\leq 1}(x_1x_2)^{a-1}((1-x_1)(1-x_2))^{b-1}
(x_1-x_2)^{2c+1}
dx_1dx_2
\end{align*}

\subsubsection{Variant}
Following the idea of Aomoto (1984) and Lavoie-Grondin-Rathie-Arora (1994), 
we consider the followoing variant.
For a polynomial $f(x,y)$ of $x,y$, we set
$$
Sel_{\Phi}(f(x,y))_{a,b,c}
=
\int_{0\leq x \leq y\leq 1}^{\Phi}f(x,y)x^{a-1}y^{a-1}(1-x)^{b-1}(1-y)^{b-1}
(y-x)^{2c}dxdy
$$
The following lemma is direct from the defintion.
\begin{lemma}
\begin{enumerate}
\item
\begin{align*}
&Sel_{\Phi}(1)_{a,b,c}=S^+_{\Phi}(a,b,c), \\
&Sel_{\Phi}(y-x)_{a,b,c}=S^-_{\Phi}(a,b,c),
\end{align*}
\item
\begin{align*}
&Sel_{\Phi}(xyf(x,y))_{a,b,c}
=Sel_{\Phi}(f(x,y))_{a+1,b,c} \\
&Sel_{\Phi}((1-x)(1-y)f(x,y))_{a,b,c}
=Sel_{\Phi}(f(x,y))_{a,b+1,c}
\end{align*}
\end{enumerate}
\end{lemma}
Using the equalities in the lemma, for any polynomial $f(x,y)$,
$Sel_{\Phi}(f(x,y))_{a,b,c}$ 
can be computed using Selberg integrals 
$S^+_{\Phi}(a,b,c)$ and $S^-_{\Phi}(a,b,c)$.

\subsection{Even Selberg integral}
In this subsection, we prove the following proposition.
\begin{proposition}[$\Phi^+$-Selberg integral formula]
\begin{align*}
S_{\Phi}^+(a,b,c)=
&\frac{
\Gamma_{\Phi}(a,b,a+c,b+c,2c)}
{\Gamma_{\Phi}(c,a+b+c,a+b+2c)}
\end{align*}
\end{proposition}

We define $\Cal A^{\Phi}_6$ module 
$$
T(a,b,c)=
x_1^ax_3^a
(1-x_1)^b
(1-x_3)^b
(x_2-x_1)^c
(x_3-x_2)^c
\bold Q[[a,b,c]]
$$
Let $\iota:\Cal M_6 \to \Cal M_6$ be an involution defined by
$(x_1,x_3)\mapsto (x_3,x_1)$.
Then we have an equivariant action $T(a,b,c)\simeq \iota_*T(a,b,c)$.
The real valued section at $p^{(0)}=(x^{(0)}_1,x^{(0)}_2,x^{(0)}_3)$ 
$(0< x_1^{(0)}<x_2^{(0)}<x_3^{(0)}<1)$ goes to the section
with real value in
\begin{align*}
\iota^*T(a,b,c)_{p^{(0)}}=&T(a,b,c)_{\iota(p^{(0)})}
\\
=&
(x_3^{(0)})^a(x_1^{(0)})^a
(1-x_3^{(0)})^b
(1-x_1^{(0)})^b
(x_2^{(0)}-x_3^{(0)})^c
(x_1^{(0)}-x_2^{(0)})^c\bold Q[[a,b,c]]
\end{align*}
We consider the pairing 
$$
H_3^{\Phi}(\Cal M_6,T(a,b,c))\otimes H^3_{\Phi}(\Cal M_6,T(-a,-b,-c))\to
\bold C[[a,b,c]].
$$

\subsubsection{The first integral formula}
Let $\Cal M_6\to \Cal M_5$ be the map defined by 
$(x_1,x_2,x_3)\mapsto (x_1,x_3)$.
\begin{align*}
&
\int^{\Phi}_{0<x_1<x_3<1}\bigg[\int^{\Phi}_{x_1<x_2<x_3}
x_1^ax_3^a
(1-x_1)^b
(1-x_3)^b
\\
&
(x_2-x_1)^c
(x_3-x_2)^c(x_3-x_1)
dx_2
\bigg]dx_1dx_3
\\
=&
\int^{\Phi}_{0<x_1<x_3<1}\bigg[
x_1^ax_3^a(1-x_1)^b
(1-x_3)^b(x_3-x_1)
\\
&
\int^{\Phi}_{x_1<x_2<x_3}
(x_3-x_2)^c
(x_2-x_1)^c
dx_2
\bigg]dx_1dx_3
\\
=&
\frac{\Gamma_{\Phi}(c+1)^2}{\Gamma_{\Phi}(2c+2)}
\int^{\Phi}_{0<x_1<x_3<1}\bigg[
x_1^ax_3^a(1-x_1)^b
(1-x_3)^b
(x_3-x_1)^{2c+2}
\bigg]dx_1dx_3
\\
=&
\frac{\Gamma_{\Phi}(c+1)^2}
{\Gamma_{\Phi}(2c+2)}S_{\Phi}(a+1,b+1,c+1)
\end{align*}
\subsubsection{The second integral formula}
We need the following formla
Let $q:\Cal M_6\to \Cal M_4$ be the map defined by
$(x_1,x_2,x_3)\mapsto x_2$.
We define $\Cal D(a,b,c)$ 
by the fixed part  
of $\bold R^2q_*(T(a,b,c))^{\iota}$. 
\begin{proposition}[Determinant formula]
\begin{enumerate}
\item
$\Cal D(a,b,c)_{dR}$ is a torsion free sheaf of rank one 
over $\bold Q[[a,b,c]]$
generated by
$$
x_1^ax_3^a
(1-x_1)^b
(1-x_3)^b
(x_2-x_1)^c
(x_3-x_2)^c(x_3-x_1)
dx_1dx_3.
$$
\item
We have an isomorphism
of $\Cal A_{4}^{\Phi}$-modules
$$
\phi:\Cal D(a,b,c)\simeq \bold B_{\Phi}(a,c)\otimes 
\bold B_{\Phi}(a+c,b)
\otimes_{\bold Q[[a,b,c]]} 
\bigg(t^{a+c}(1-t)^{b+c}\bold Q[[a,b,c]]\bigg)
$$
\end{enumerate}
\end{proposition}
\begin{proof}
(2) We consider the $\Cal A^{\Phi}_4$-module $M$ defined by
$$
M=\Cal D(a,b,c)
\otimes_{\bold Q[[a,b,c]]} 
\bigg(t^{-a-c}(1-t)^{-b-c}\bold Q[[a,b,c]]\bigg)
$$
We compute the action of $A_{dR}^4$ on $M_{dR}$.
Then the map
$$
M_{dR}\xrightarrow{e_0,e_1} M_{dR}\oplus M_{dR}
$$
is the zero map. Therefore $M$ is the pull back of
an object in $\Cal C$. 
To consider the fiber at $\overline{01}$,
we consider the integral
\begin{align*}
&
x_2^{-a-c-1}\int^{\Phi}_{0\leq x_1\leq t,t\leq x_3 \leq 1}
x_1^{a}x_3^{a}
(1-x_1)^b
(1-x_3)^b
(x_2-x_1)^c
(x_3-x_2)^c(x_3-x_1)
dx_1dx_3
\\
=&
\int^{\Phi}_{0\leq \xi\leq 1,x_2\leq x_3 \leq 1}
\xi^{a}x_3^{a}
(1-x_2\xi)^b
(1-x_3)^b
(1-\xi)^c
(x_3-x_2)^c(x_3-x_2\xi)
d\xi dx_3
\\
\underset{x_2 \to 0}\to&
\int^{\Phi}_{0\leq \xi\leq 1,0\leq x_3 \leq 1}
\xi^{a}x_3^{a+c+1}
(1-x_3)^b
(1-\xi)^c
d\xi dx_3
\\
\end{align*}
Therefore the fiber 
$M_{\overline{01}}\in \Cal C$ is isomprhic to 
$\bold B_{\Phi}(a,c)\otimes \bold B_{\Phi}(a+c,b)$.
Thus we have the proposition.
\end{proof}
Using the above proposition, we have
\begin{align*}
&
\int^{\Phi}_{0<x_2<1}
\bigg[
\int^{\Phi}_{0<x_1<x_2,x_2<x_3<1}
x_1^{a}x_3^{a}
(1-x_1)^b
(1-x_3)^b
\\
&
(x_2-x_1)^c
(x_3-x_2)^c
(x_3-x_1)dx_1dx_3
\bigg]dx_2
\\
=&
\frac{\Gamma_{\Phi}(a+1,c+1,a+c+2,b+1)}
{\Gamma_{\Phi}(a+c+2,a+b+c+3)}
\int^{\Phi}_{0<x_2<1}
x_2^{a+c+1}(1-x_2)^{b+c+1}
dx_2
\\
=&
\frac{\Gamma_{\Phi}(a+1,c+1,b+1,a+c+2,b+c+2)}
{\Gamma_{\Phi}(a+b+c+3,a+b+2c+4)}
\end{align*}

\subsubsection{Proof of $\Phi$-Selberg integral formual}
By computing the integral
\begin{align*}
&\int^{\Phi}_{0<x_1<x_2<x_3<1}x_1^ax_3^a
(1-x_1)^b
(1-x_3)^b
\\
&
(x_2-x_1)^c
(x_3-x_2)^c
(x_3-x_1)
dx_1dx_2dx_3
\end{align*}
in two ways, we have
\begin{align*}
&\frac{\Gamma(a+1)\Gamma(c+1)\Gamma(b+1)\Gamma(a+c+2)\Gamma(b+c+2)}
{\Gamma(a+b+c+3)\Gamma(a+b+2c+4)} \\
=&\frac{\Gamma_{\Phi}(c+1)^2}
{\Gamma_{\Phi}(2c+2)}S^+_{\Phi}(a+1,b+1,c+1)
\end{align*}
and
\begin{align*}
\frac{\Gamma_{\Phi}(2c+2,a+1,b+1,a+c+2,b+c+2)}
{\Gamma_{\Phi}(c+1,a+b+c+3,a+b+2c+4)}
=
S^+_{\Phi}(a+1,b+1,c+1)
\end{align*}

\subsection{Odd $\Phi$-Selberg integral}
We compute the following odd $\Phi$-Selberg integral $S^-_{\Phi}(a,b,c)$

First we consider the differential equations satisfied by
$\ _3F_2$.
For a cycle $\gamma$, we consider the integral
$$
f(\gamma)=\int^{\Phi}_{\gamma}t_1^{a_1}(1-t_1)^{c_1-a_1}t_2^{a_2}(1-t_2)^{c_2-a_2}
(1-xt_1t_2)^{-a_3}\frac{dt_1}{t_1}\frac{dt_2}{t_2}
$$
We change coordinates to $(t_2,t_3)$ with the relation
$xt_1t_2t_3=1$. It is also expressed by similar integral expression.
Using this expression,
\begin{align*}
f_{11}&=r_1F_{\Phi}(a_1,a_2,a_3;c_1+1,c_2+1;x) \\
f_{12}&=r_2x^{-c_1}F_{\Phi}(a_2-c_1,a_3-c_1,a_1-c_1;c_2-c_1+1,-c_1+1;x) \\
f_{13}&=r_3x^{-c_2}F_{\Phi}(a_1-c_2,a_3-c_2,a_2-c_2;c_1-c_2+1,-c_2+1;x)
\end{align*}
with
\begin{align*}
&r_1=B_{\Phi}(a_1,c_1-a_1+1)B_{\Phi}(a_2,c_2-a_2+1) \\
&r_2=B_{\Phi}(a_2-c_1,c_2-a_2+1)B_{\Phi}(a_3-c_1,-a_3+1)\\
&r_3=B_{\Phi}(a_1-c_2,c_1-a_1+1)B_{\Phi}(a_3-c_2,-a_3+1)\\
\end{align*}
satisfies the same rational differential equation of $t$.
Thus we have
\begin{align*}
r_1r_2r_3=
&\frac{\Gamma_{\Phi}(a_1,c_1-a_1+1,a_2,c_2-a_2+1)}
{\Gamma_{\Phi}(c_1+1,c_2+1)}
\\
&\frac{\Gamma_{\Phi}(a_2-c_1,c_2-a_2+1,a_3-c_1,-a_3+1)}
{\Gamma_{\Phi}(-c_1+c_2+1,-c_1+1)}
\\
&\frac{\Gamma_{\Phi}(a_1-c_2,c_1-a_1+1,a_3-c_2,-a_3+1)}
{\Gamma_{\Phi}(-c_2+c_1+1,-c_2+1)}
\\
\end{align*}
They are integral of the following chains up to constant:
$$
\gamma_1=\{(t_1,t_2)\in [0,1]^2\},\quad
\gamma_2=\{(t_2,t_3)\in [0,1]^2\},\quad
\gamma_3=\{(t_1,t_3)\in [0,1]^2\}.
$$
We define a matrix $F=(\bold f_1,\bold f_2,\bold f_3)=(f_{ij})$ by
and
$$
f_{2i}=x\frac{df_{1i}}{dx},\quad
f_{3i}=x\frac{df_{2i}}{dx}.
$$
In general, we set
$$
\bold f(\gamma)=
\left(
\begin{matrix}
f(\gamma) \\
x\frac{d}{dx}f(\gamma) \\
x\frac{d}{dx}x\frac{d}{dx}f(\gamma)
\end{matrix}
\right),\quad
\bold f'(\gamma)=
\lim_{t\to 1}\left(
\begin{matrix}
f(\gamma) \\
x\frac{d}{dx}f(\gamma) 
\end{matrix}
\right)
$$

By setting
\begin{align*}
P=&
\frac{dx}{x}
\left(\begin{matrix}0 & 1 & 0\\ 0 & 0 & 1\\ 0 & -c_2c_1 & -c_2-c_1
\end{matrix}\right) \\
&+\frac{dx}{x-1}
\left(\begin{matrix}
0 & 0 & 0 \\ 0 & 0 & 0 \\
-a_1a_2a_3 & -a_2a_1+c_2c_1-a_3a_1-a_3a_2 & -a_1-a_2-a_3+c_1+c_2
\end{matrix}\right),
\end{align*}
we have
$$
dF=PF.
$$
Therefore
$$
\det(F)=cx^{-c_1-c_2}(1-x)^{-a_1-a_2-a_3+c_1+c_2}
$$
with some constant $c$. By considering the limit for $x=0$, we have
$$c=r_1r_2r_3c_1c_2(c_1-c_2).$$
Now we consider the limit for $x=1$.
Changing coordinate 
$\xi_1=\frac{1}{t_1}$,
$\xi_2=t_2$, we have
\begin{align*}
&\int^{\Phi}_{\gamma}t_1^{a_1}(1-t_1)^{c_1-a_1}t_2^{a_2}(1-t_2)^{c_2-a_2}
(1-xt_1t_2)^{-a_3}\frac{dt_1}{t_1}\frac{dt_2}{t_2} \\
=&\int^{\Phi}_{\gamma}\xi_1^{a_3-c_1}(\xi_1-1)^{c_1-a_1}\xi_2^{a_2}(1-\xi_2)^{c_2-a_2}
(\xi_1-x\xi_2)^{-a_3}\frac{d\xi_1}{\xi_1}\frac{d\xi_2}{\xi_2}
\end{align*}
Under this coordinate, $\gamma_2$ is equal to
$$
\gamma_2=\{0\leq \xi_1\leq x\xi_2, 0\leq \xi_2 \leq 1\}.
$$
We set
\begin{align*}
&\gamma_5=\{x\xi_2\leq \xi_1 \leq 1, 0\leq \xi_2\leq 1\} \\
&\tau=\{x\xi_2\leq \xi_1 \leq 1, 1\leq \xi_2\leq \frac{1}{x}\} 
\end{align*}
Using the equality
\begin{align*}
&
\sin(\pi\alpha)\int_0^u
x^{\alpha}
(u-x)^{\beta}
(1-x)^{\gamma}
\frac{dx}{x} 
+\sin(\pi(\alpha+\beta))\int_u^1
x^{\alpha}
(x-u)^{\beta}
(1-x)^{\gamma}
\frac{dx}{x} \\
+&\sin(\pi(\alpha+\beta+\gamma))\int_1^{\infty}
x^{\alpha}
(x-u)^{\beta}
(x-1)^{\gamma}
\frac{dx}{x}=0
\end{align*}
we have the following relations for topological cycles.
\begin{lemma}
We have
\begin{align*}
&
\gamma_5=
-\frac{\sin(\pi a_1)}{\sin(\pi c_1)}\gamma_1
+\frac{\sin(\pi (a_3-c_1))}{\sin(\pi c_1)}\gamma_2 \\
&
\tau=k_1\gamma_1+k_2\gamma_2+k_3\gamma_3
\end{align*}
with
$$
k_3=
\frac{\sin(\pi(c_2-a_1))\sin(\pi(c_2-a_3))}
{\sin(\pi c_2)\sin(\pi(c_1-c_2))}.
$$
As a consequence, we have
$$
\det(\bold f(\gamma_2),\bold f(\gamma_5),\bold f(\tau))=
Cx^{-c_1-c_2}(1-x)^{c_1+c_2-a_1-a_2-a_3}
$$
with
\begin{align*}
C=
&
\Gamma_{\Phi}(a_2,c_1-a_1+1,c_2-a_2+1,c_2-a_2+1,-a_3+1)
 \\
&
\frac{\Gamma_{\Phi}(a_2-c_1,c_1-a_1+1,a_3-c_1,-a_3+1)}
{\Gamma_{\Phi}(-a_1+1,c_2-a_1+1,c_2-a_3+1)} 
\end{align*}
\end{lemma}
By taking a limit $t \to 1$, since $\tau$ is a vanishing cycle,
we have the following theorem.
\begin{proposition}
\label{theorem minus part of Selberg}
\begin{enumerate}
\item
\begin{align*}
&\lim_{t\to 1}(1-t)^{-c_1-c_2+a_1+a_2+a_3}\bold f(\tau)
=\left(
\begin{matrix}
0 \\ 0 \\
J
\end{matrix}
\right),
\\
&J=\frac
{\Gamma_{\Phi}(-a_3+1,c_1-a_1+1,c_2-a_2+1)}
{\Gamma_{\Phi}(c_1+c_2-a_1-a_2-a_3+1)}
\end{align*}
\item
\begin{align*}
&\det (\bold f'(\gamma_2),\bold f'(\gamma_5)) \\
=&
\Gamma_{\Phi}(a_2,c_1+c_2-a_1-a_2-a_3+1,a_2-c_1,c_1-a_1+1,a_3-c_1)
\\
&
\frac{\Gamma_{\Phi}(c_2-a_2+1,-a_3+1)
}{\Gamma_{\Phi}(-a_1+1,c_2-a_1+1,c_2-a_3+1)} \\
\end{align*}
\end{enumerate}
\end{proposition}
\begin{corollary}
\end{corollary}
\begin{proof}
We use Proposition \ref{theorem minus part of Selberg}, 
by settnig
$-a_3-1=2c,c_1-a_1=c_2-a_2=b-1,a_2=a_3-c_1=a$
hand have
\begin{align*}
&(2c+1)S_{\Phi}^-(a,b,c)Sel_{\Phi}(x+y)_{a,b,c} 
\\
=&\det\left(
\begin{matrix}
S_{\Phi}^-(a,b,c) & -S_{\Phi}^-(a,b,c) \\
(2c+1)Sel_{\Phi}(x)_{a,b,c} & -(2c+1)Sel_{\Phi}(y)_{a,b,c}
\end{matrix}
\right) \\
=&
\frac{\Gamma_{\Phi}(a,a,b,b,2c+2b,2a+2c+1,2c+2)}
{\Gamma_{\Phi}(2c+a+b+1,2c+a+b+1,2a+2b+2c)}
\end{align*}
Since
\begin{align*}
Sel_{\Phi}(x+y)_{a,b,c}=&
Sel_{\Phi}(1)_{a+1,b,c}-
Sel_{\Phi}(1)_{a,b+1,c}+
Sel_{\Phi}(1)_{a,b,c} \\
=&
\frac{2\Gamma_{\Phi}(a+1+c,2c,b+c,b,a)}
{\Gamma_{\Phi}(c,a+1+b+2c,a+b+c)}
\end{align*}
We have
\begin{align*}
&S_{\Phi}^-(a,b,c) \\
=&
\frac
{\Gamma_{\Phi}(a,b,c+1,a+b+c,2c+2b,2a+2c+1)}
{\Gamma_{\Phi}(2a+2b+2c,b+c,a+c+1,2c+a+b+1)
} \\
\end{align*}
\end{proof}
\subsection{$\Phi$-Dixon's theorem and its variants}
We prove Dixon's theorem and its generalization using
Selberg integral formula.
\begin{proposition}[Dixson's theorem and its variant]
\begin{enumerate}
\item
\begin{align*}
&F_{\Phi}(2c,b,a;2c-a+1,2c-b+1;1) 
\\=&
\frac{\Gamma_{\Phi}(1+c,1+2c-a,1+2c-b,1+c-a-b)}
{\Gamma_{\Phi}(1+2c,1+c-a,1+c-b,1+2c-a-b)}
\end{align*}
\item
\begin{align*}
&F_{\Phi}(2c+1,b,a;2c-a+3,2c-b+3;1) \\
=&\bigg(
\Gamma_{\Phi}(2c,2c+4-2a-2b,c+1-b,c+1-a,c+2-a,c+2-b)
\\
&-\Gamma_{\Phi}(2c+2-2b,2c+2-2a,c,c+2-a-b,c+3-a-b,c+1)
\bigg)
\\
&
\frac{\Gamma_{\Phi}(2c-b+3,2c-a+3)}
{
\substack{
2(b-1)(a-1)\Gamma_{\Phi}(2c+3-a-b,c+2-a,c+2-b,2c+2-2b,2c+2-2a,
c,c+2-a-b,2c+1)
}
}
\end{align*}
\item
\begin{align*}
&F_{\Phi}(2c+1,1,a;2c-a+3,2c+2;1) \\
=&
-\frac{(2c+1)(a-2c-2)
(\Psi(2c+1)-\Psi(c+1)
-\Psi(2c+3-2a)+\Psi(c+2-a))}
{(a-1)}
\end{align*}
Here we define 
$$
\Psi(x)=\frac{d}{dx}\log\Gamma_{\Phi}(z).
$$
\end{enumerate}
\end{proposition}
\begin{proof}
\begin{align*}
&
\frac{
\Gamma_{\Phi}(2c,-a+1,b,2c-2b+1)}
{\Gamma_{\Phi}(2c-a+1,2c-b+1)}
F_{\Phi}(2c,b,a;2c-a+1,2c-b+1;1) \\
=&
\int_{[0,1]^2}^{\Phi}x^{2c-1}(1-x)^{-a}y^{b-1}(1-y)^{2c-2b}(1-xy)^{-a}dxdy \\
=&
\int^{\Phi}_{0\leq t \leq x\leq 1}x^{b-1}(1-x)^{-a}t^{b-1}
(x-t)^{2c-2b}(1-t)^{-a}dxdt \\
=&S^+_{\Phi}(b,-a+1,c-b)
\end{align*}
Here we change variables by $(x,t)\to (x,y)=(x,t/x)$.
By Selberg integral formula, we have the proposition.

(2)
Using even and odd Selberg integrals, we have
\begin{align*}
&
\frac{\Gamma_{\Phi}(2c+1,-a+2,b,2c-2b+3)}
{\Gamma_{\Phi}(2c-a+3,2c-b+3)}
F_{\Phi}(2c+1,b,a;2c-a+3,2c-b+3;1) \\
=&
\int_{[0,1]^2}^{\Phi}x^{2c}(1-x)^{-a+1}y^{b-1}(1-y)^{2c-2b+2}(1-xy)^{-a}dxdy \\
=&
\int^{\Phi}_{0\leq t \leq x\leq 1}
x^{b-2}(1-x)^{-a+1}t^{b-1}(x-t)^{2c-2b+2}(1-t)^{-a}dxdt
\\
=&
Sel_{\Phi}((1-x)y)_{b-1,-a+1,c-b+1} 
\end{align*}
Therefore we have the proposition

(3)
We take a limit $b\to 1$.
\begin{align*}
&F_{\Phi}(2c+1,1,a;2c-a+3,2c+2;1) \\
=&
\frac{1}{2}
\lim_{b\to 1}
Sel_{\Phi}(1-(1-x)(1-y)-(x-y)-xy)_{b-1,-a+1,c-b+1} \\
=&
-\frac{(2c+1)(a-2c-2)
(\Psi(2c+1)-\Psi(c+1)
-\Psi(2c+3-2a)+\Psi(c+2-a))}
{(a-1)}
\end{align*}
\end{proof}

\section{Brown-Zagier relation for associators}
\subsection{Li's computation for Brown-Zagier relation}
We have the folloing relations between 
$F_{\Phi}(a_1,a_1,a_3;b_1,b_2)=
F_{\Phi}(a_1,a_1,a_3;b_1,b_2;1)$
arising from relations in 
$H^{\Phi,B}_{2}(\Cal M_5,\Cal F(a_1,a_2,a_3;b_1,b_2))$
and $H^2_{\Phi,dR}(\Cal M_5, \Cal F(a_1,a_2,a_3;b_1,b_2))$.
\begin{proposition}
\label{several relations for hypergeometric constatns}
We have the following equalities
\begin{enumerate}
\item
\label{1st relation of hypergeom const}
\begin{align*}
&F_{\Phi}(x,-x,z;1+y,1-y)
\\=
&\frac{1}{2}
F_{\Phi}(x,1-x,z;1+y,1-y)+
\frac{1}{2}
F_{\Phi}(1+x,-x,z;1+y,1-y)
\end{align*}
\item
\label{2nd relation of hypergeom const}
\begin{align}
\label{dividing two terms}
&F_{\Phi}(x,1-x,z;1+y,1-y)
\\=
\nonumber
&\frac{\Gamma_{\Phi}(1+y,1-x+y-z)}{\Gamma_{\Phi}(1-x+y,1+y-z)}
F_{\Phi}(x,x-y,z;x-y+z,1-y) \\
\nonumber
&+
\frac{\Gamma_{\Phi}(1+y,1-y,x-y+z-1,1-z)}
{\Gamma_{\Phi}(x,z,x-y,2-y-z)}\\
\nonumber
&
F_{\Phi}(1-x+y,1+y-z,1-z;2-x+y-z,2-x-z) 
\end{align}
\item
\label{3rd relation of hypergeom const}
\begin{align*}
&F_{\Phi}(x,x-y,z;x-y+z,1-y) \\
=&
\frac{\Gamma_{\Phi}(1-x-y,1-y)}{\Gamma_{\Phi}(1-y-z,1-x-y+z)}
F_{\Phi}(-y+z,z,z;x-y+z,1-x-y+z)
\end{align*}
\end{enumerate}
\end{proposition}
\begin{proof}
The equality (1) follows from an equality in de Rham cohomology.
The equality (2) follows from an equality for Betti cohomology
and change of coordinates.
The equality (3) follows by changing coordinate of integral expression.
\end{proof}
Since
$$
F_{\Phi}(a_1,a_2,a_3;c_1,c_2)\in 1+a_1\bold C[[a_1, \dots, c_2]],
$$
we have the following proposition.
\begin{lemma}
\label{vanishing differential}
$$
\frac{d}{dz}F_{\Phi}(-y+z,z,z;x-y+z,1-x-y+z)\mid_{z=0}=0
$$
\end{lemma}
Now we are ready to compute the function in Theorem
\ref{first identity by Zagier} using Li's computation.
\begin{proposition}
\label{relation to psi functions}
We have
\begin{align*}
&\frac{d}{dz}F_{\Phi}(x,1-x,z;1+y,1-y) \\
=&
\Psi(1+y)+\Psi(1-y)-\Psi(1-x+y)-\Psi(1-x-y) \\
&-\frac{\bold s(x)}{\bold s(y)}
(\Psi(1-x+y)-\Psi(1-x-y)-\Psi(1-\frac{x-y}{2})+\Psi(1-\frac{x+y}{2}))
\end{align*}
\end{proposition}
\begin{proof}
We compute the each term of the derivative of (\ref{dividing two terms}).
Using Proposition 
\ref{several relations for hypergeometric constatns} 
(\ref{3rd relation of hypergeom const}) and
Lemma \ref{vanishing differential}, we have
\begin{align*}
&
\frac{d}{dz}\bigg(\frac{\Gamma_{\Phi}(1+y,1-x+y-z)}{\Gamma_{\Phi}(1-x+y,1+y-z)}
F_{\Phi}(x,x-y,z;x-y+z,1-y)\bigg)_{z=0} \\
=&
\frac{d}{dz}\bigg(\frac{\Gamma_{\Phi}(1+y,1-y,1-x-y,1-x+y-z)}
{\Gamma_{\Phi}(1-x+y,1+y-z,1-y-z,1-x-y+z)} \\
&F_{\Phi}(-y+z,z,z;x-y+z,1-x-y+z)\bigg)_{z=0} \\
=&
\Psi(1+y)+\Psi(1-y)-\Psi(1-x+y)-\Psi(1-x-y)
\end{align*}

We compute the derivative of the second term of (\ref{dividing two terms}).
Since $\lim_{z\to 0}\Gamma_{\Phi}(z)z=1$, we have
\begin{align*}
&\frac{d}{dz}\bigg(\frac{\Gamma_{\Phi}(1+y,1-y,x-y+z-1,1-z)}
{\Gamma_{\Phi}(x,z,x-y,2-y-z)}\\
\nonumber
&
F_{\Phi}(1-x+y,1+y-z,1-z;2-x+y-z,2-x-z) \bigg)_{z=0}
\\
=
&\frac{\Gamma_{\Phi}(1+y,1-y,x-y-1)}
{\Gamma_{\Phi}(x,x-y,2-x)}
F_{\Phi}(1-x+y,1+y,1;2-x+y,2-x) \\
=
&\frac{y\bold s(x)}
{(x-1)(x-y-1)\bold s(y)}
F_{\Phi}(1-x+y,1+y,1;2-x+y,2-x) 
\end{align*}
By setting $a=y+1,2c=y-x$ in the equality of
Proposition 
\ref{several relations for hypergeometric constatns} 
(\ref{2nd relation of hypergeom const}),
it is equal to
\begin{align*}
&\frac{\bold s(x)}
{\bold s(y)}
(\Psi(y-x+1)-\Psi(1+\frac{y-x}{2})-\Psi(1-x-y)-\Psi(1-\frac{x+y}{2}))
\end{align*}
Thus we have the proposition.
\end{proof}
\begin{proof}[Proof of Theorem \ref{main theorem}]

By Theorem \ref{first identity by Zagier},
and Proposition 
\ref{several relations for hypergeometric constatns} 
(\ref{1st relation of hypergeom const}) and Proposition 
\ref{relation to psi functions}, we have
\begin{align*}
&\sum_{n\geq 0,m>0}c_{\Phi,(01)^n0(01)^m}y^{2n+1}x^{2m} \\
=&\bold s(y)
\frac{d}{dz}\mid_{z=0}F_{\Phi}(x,-x,z;1+y,1-y;1) \\
=&
\frac{\bold s(y)}{2}
\frac{d}{dz}\mid_{z=0}
\bigg(
F_{\Phi}(x,1-x,z;1+y,1-y)+
F_{\Phi}(1+x,-x,z;1+y,1-y)
\bigg) \\
=&
\frac{\bold s(y)}{2} 
\bigg(
2\Psi(1+y)+2\Psi(1-y)-\Psi(1+x+y) \\
&-\Psi(1-x-y)-\Psi(1+x-y)-\Psi(1-x+y))
\bigg)
\\
&-\frac{\bold s(x)}{2}
\bigg(
\Psi(1+\frac{x+y}{2})
+\Psi(1-\frac{x+y}{2})
-\Psi(1+\frac{x-y}{2})
-\Psi(1+\frac{y-x}{2})
 \\
&-\Psi(1+x+y)-\Psi(1-x-y)+\Psi(1+x-y)+\Psi(1-y-x))
\bigg)
\end{align*}
Using the equality
$$
\Psi(1+z)=\frac{d}{dx}\log\Gamma_{\Phi}(z+1)=\sum_{n=2}^{\infty}(-1)^n\zeta_{\Phi}(n)x^{n-1}
$$
we have the theorem.
\end{proof}

\end{document}